\documentclass[11pt,reqno]{amsart}
\usepackage{amssymb,amsmath,amscd}
\usepackage{latexsym,bm,bbm,mathrsfs}
\usepackage{hyperref,graphicx, enumerate}
\usepackage{mathtools}
\usepackage{stmaryrd}
\usepackage[all,pdf]{xy}
\usepackage{graphicx}
\usepackage{caption}
\usepackage{tikz-cd}
\usepackage{color}
\graphicspath{ {./} }

\usepackage{stackrel}
\usepackage[left=30mm, right=30mm, top=30mm, bottom=30mm]{geometry}
\usepackage[shortlabels]{enumitem}
\usepackage{ulem}
\usepackage{verbatim}

\allowdisplaybreaks

\newtheorem{thm}{Theorem}[section]

\newtheorem*{quest*}{Question}

\newtheorem{theorem}[thm]{Theorem}
\newtheorem{cor}[thm]{Corollary}
\newtheorem{prop}[thm]{Proposition}
\newtheorem{lemma}[thm]{Lemma}

\newtheorem{remark}[thm]{Remark}
\newtheorem{defn}[thm]{Definition}

\newtheorem*{theorem*}{Theorem}

\newcommand{\tors}{\operatorname{tors}}

\newcommand{\bba}{{\mathbb{A}}}
\newcommand{\bbc}{{\mathbb{C}}}
\newcommand{\bbp}{{\mathbb{P}}}
\newcommand{\bbq}{{\mathbb{Q}}}

\newcommand{\bbz}{{\mathbb{Z}}}

\newcommand{\cB}{\mathcal{B}}
\newcommand{\cF}{\mathcal{F}}
\newcommand{\cH}{\mathcal{H}}
\newcommand{\cP}{\mathcal{P}}

\newcommand{\scrE}{\mathscr{E}}

\newcommand{\scrU}{\mathscr{U}}

\newcommand{\fkR}{\mathfrak{R}}

\newcommand{\Gbar}{\overline{G}}
\newcommand{\rbar}{\overline{r}}
\newcommand{\Hbar}{\overline{H}}
\newcommand{\SL}[1]{\operatorname{SL}_{2}\left(\bbz/{#1}\bbz\right)}
\newcommand{\GL}[1]{\operatorname{GL}_{2}\left(\bbz/{#1}\bbz\right)}

\newcommand{\Gal}{\operatorname{Gal}}
\renewcommand{\Im}{\operatorname{Im}}
\newcommand{\Gm}{\mathbb{G}_{\operatorname{m}}}

\title[Density of elliptic curves  with prescribed torsion subgroups]{Density of elliptic curves over number fields with prescribed torsion subgroups}

\author{Bo-Hae Im}
\address{Department of Mathematical Sciences, KAIST, 291 Daehak-ro, Yuseong-gu, Daejeon, 34141, South Korea}
\email{bhim@kaist.ac.kr}
\thanks{Bo-Hae Im was supported by Basic Science Research Program through the National Research Foundation of Korea(NRF) grant funded by the Korea government(MSIT)(NRF-2023R1A2C1002385).}

\author{Hansol Kim}
\address{Institute of Mathematics, Academia Sinica, 6F, Astronomy-Mathematics Building, No. 1, Sec. 4, Roosevelt Road, Taipei 10617, Taiwan}
\email{jawlang@gate.sinica.edu.tw}
\thanks{Hansol Kim was supported by NSTC grant funded by Taiwan government (MST) (No.~113-2811-M-001-068).}

\date{\today}

\begin{document}

\maketitle

\begin{abstract}
Let $K$ be a number field. For positive integers $m$ and $n$ such that $m\mid n$, we define $\mathscr{S}_{m,n}$ as the set of elliptic curves $E/K$ defined over $K$ such that $E(K)_{\tors}\supseteq \mathscr{T}\cong \bbz/m\bbz\times \bbz/n\bbz$.  We prove that if  the genus of the modular curve $X_{1}(m,n)$ is $0$, then `almost all' $E\in \mathscr{S}_{m,n}$ satisfy that $E(K)_{\tors}= \mathscr{T}$, i.e., $E(K)_{\tors}$ is not larger than $\mathscr{T}$. In particular, when $m=n=1$, our result generalizes the works of Duke (\cite{Duke}) and Zywina (\cite{Zywina}) over number fields $K$ satisfying $K \cap \bbq^{\operatorname{cyc}} = \bbq$, extending their results to arbitrary number fields~$K$ for the trivial torsion subgroup.
\end{abstract}

\section{Introduction}\label{intro}

The group structure of an elliptic curve over various fields and its application have been actively studied as one of the main topic in the area of number theory. Especially, for an elliptic curve $E$ over a number field $K$, it is known by the Mordell-Weil theorem (\cite[VIII.Theorem~6.7]{Silverman}) that  the set of $K$-rational points $E(K)$ of $E$ called {\it the Mordell-Weil group} is a finitely generated abelian group, i.e., $E(K)\cong E(K)_{\tors}\times \bbz^r$, where  $E(K)_{\tors}$ is its  finite torsion subgroup and an integer  $r\geq 0$ is  called {\it the rank} of $E(K)$ which is the number of independent non-torsion points in $E(K)$.  But there hasn't been yet found an efficient algorithm to compute the rank $r$ of $E(K)$, and there have been only some partial results on  the finite abelian group structure or the order of $E(K)_{\tors}$.

If we introduce several known results on the abelian group structure of $E(K)_{\tors}$, Mazur \cite{Mazur}  has classified all realizable torsion subgroups of elliptic curves over~$\bbq$ as follows:

\begin{theorem}[\cite{Mazur}]\label{thm:org_Mazur}
For an elliptic curve $E/\bbq$ defined over $\bbq$,
	the torsion subgroup $E(\bbq)_{\tors}$ of $E(\bbq)$ is isomorphic to one of the following groups:
	\begin{itemize}
		\item $\bbz/m\bbz$ where $m\in \{1,2,3,4,5,6,7,8,9,10,12\}$; \text{or }
		\item $\bbz/2\bbz\times \bbz/2n\bbz$ where $n\in \{1,2,3,4\}$.
	\end{itemize}
	Conversely, each group listed above can be realizable as a torsion subgroup $E(\bbq)_{\tors}$ for infinitely many (non-isomorphic) elliptic curves $E$ defined over $\bbq$. \\
\end{theorem}

Mazur's result, Theorem~\ref{thm:org_Mazur}	implies the following, which is obvious but worth mentioning for our interest in this paper.
\begin{cor}\label{cor:maxQ}  
	Each of the maximal classes of the torsion subgroups $E(\bbq)_{\tors}$ of $E(\bbq)$ For elliptic curves $E$ over $\bbq$ is isomorphic to  one of the following groups:
	\begin{itemize}
		\item $\bbz/m\bbz$ where $m\in \{7,9,10,12\}$; \text{or }
		\item $\bbz/2\bbz\times \bbz/2n\bbz$ where $n\in \{3,4\}$.
	\end{itemize}
	Conversely, each group $\mathscr{T}$  listed above is realizable as a torsion subgroup of an elliptic curve defined over $\bbq$, and if there is an elliptic curve $E$ over $\bbq$ with $E(\bbq)_{\tors}\cong \mathscr{T}$, then for any finite  subgroup $\mathscr{T}'\subseteq \mathscr{T}$, there are infinitely many (non-isomorphic) elliptic curves $E'$ over $\bbq$ such that  $E'(\bbq)_{\tors}\cong\mathscr{T}'$.
\end{cor}

It is natural to ask how many elliptic curves have each of possibly realizable group structures as their torsion subgroups over a given number field. In order to introduce the related questions in more detail, along with  related known results and our main result of this paper, we first give some notations and definitions which will be used throughout this paper:
\begin{itemize}
	\item $K$: a number field.
	\item $\scrU:=\{(s,t)\in \bba^{2}:4s^{3}+27t^{2}\ne 0\}$: a rational variety in $\bba^{2}$.
	\item $E_{s,t}$: an  elliptic curve whose affine equation is given by $ y^{2}=x^{3}+sx+t$ with coefficients~$s, t$ in the function field $K(\scrU)=K(s,t)$ of $\scrU$.
	\item $E_{A,B}$: the specialization of  $E_{s,t}$ at $\left(A,B\right)\in \bba^{2}(K)$. 
	\item ${E_{A,B}(K)}_{\tors}$: the torsion subgroup of the Mordell-Weil group $E_{A,B}(K)$ of $E_{A,B}$ over~$K$.
	\item $H$: the absolute height on $\bbp^{n}(K)$ (see \cite[VIII.5]{Silverman}).
	\item $\cH\left(A,B\right):=H(A^{3},B^{2})$, the naive height of the elliptic curve $E_{A,B}$.
	\item $N_{\mathcal{S}}(X)$: the number of points $P\in\mathcal{S}\subseteq \bbp^{n}(K)$ such that $H(P)\le X$ for a positive real number $X$.
\end{itemize}
Note that the specialization $E_{A,B}$ is an elliptic curve if and only if $\left(A,B\right)\in \scrU(K)$.
\begin{defn}\label{defn:almostall}  	For two subsets $\mathcal{T}$ and $\mathcal{S}$ of $\bbp^{n}(K)$ such that  $\mathcal{T}\subseteq \mathcal{S}$, we say that  almost all points $P\in \mathcal{S}$ belong to  $\mathcal{T}$, if $$
		\lim_{X\to \infty}\frac{N_{\mathcal{T}}(X)}{N_{\mathcal{S}}(X)}=1.
	$$
\end{defn}

In this paper, for any number field $K$, we are interested in  the density of elliptic curves over~$K$ whose torsion subgroups~over $K$ are isomorphic to each of realizable torsion subgroups over $K$, by counting such  elliptic curves up to height of the coefficients of their affine equations.
	
First, Harron and Snowden \cite{Harron_Snowden} have observed how often each torsion subgroup over $\bbq$ listed in Theorem~\ref{thm:org_Mazur} appears:

\begin{theorem}[\cite{Harron_Snowden}]\label{thm:old_Harron_Snowden}
	For each finite abelian group $\mathscr{T}$ listed in Theorem~\ref{thm:org_Mazur}, there is a constant $d_{\mathscr{T}} > 0$ such that 
		$$\frac{1}{d_{\mathscr{T}}}
		=\lim_{X\to \infty}
		\frac
		{\log \#\{\left(A,B\right)\in \scrU(\bbq):\cH\left(A,B\right)\le X, {E_{A,B}(\mathbb{ Q })}_{\tors}\cong \mathscr{T}\}}
		{\log \#\{\left(A,B\right)\in \scrU(\mathbb { Q }):\cH\left(A,B\right)\le X\}}.
	$$
	Moreover, we have that $d_{\mathscr{T}}<d_{\mathscr{T}'}$, for any finite abelian group $\mathscr{T}'$ listed in Theorem~\ref{thm:org_Mazur} 
	such that $\mathscr{T}\subsetneq \mathscr{T}'$
\end{theorem}

Harron and Snowden \cite{Harron_Snowden} have computed the exact value of $d_{\mathscr{T}}$ for each possible torsion subgroup $\mathscr{T}$ over $\bbq$ of an elliptic curve defined  over $\bbq$. On the other hand, if we take the asymptotic behavior of the density as $\mathscr{T}$ gets bigger and so  as $d_{\mathscr{T}}$ increases strictly, then Theorem~\ref{thm:old_Harron_Snowden} implies the following:

\begin{cor}\label{cor:Harron_Snowden}
	For any finite abelian group $\mathscr{T}$ listed in Theorem~\ref{thm:org_Mazur}, almost all elliptic curves over $\bbq$ containing $\mathscr{T}$ have torsion subgroup $\mathscr{T}$ over $ \mathbb{ Q } $ exactly, i.e., $$
		\lim_{X\to \infty}\frac{\#\{\left(A,B\right)\in \scrU(\bbq):\cH\left(A,B\right)\le X, {E_{A,B}(\mathbb{ Q })}_{\tors}\cong \mathscr{T}\}}{\#\{\left(A,B\right)\in \scrU(\mathbb { Q }):\cH\left(A,B\right)\le X, {E_{A,B}(\mathbb { Q })}_{\tors}\supseteq \mathscr{T}\}}=1.
	$$
\end{cor}

Motivated by Harron-Snowden's  result over $\bbq$ ~\cite{Harron_Snowden} and extending it over arbitrary number fields, we raise the following questions.

\begin{quest*} Let $K$ be a number field.
	\begin{enumerate}[\normalfont Q1]
		\item \label{que1}\hspace{-.15cm}. Assume that there exists an elliptic curve  over $K$ whose torsion subgroup over $K$ is isomorphic to a finite abelian group $\mathscr{T}$. Then for any finite abelian subgroup $\mathscr{T}'\subseteq \mathscr{T}$, is there an elliptic curve $E$ over $K$ such that  $E(K)_{\tors}\cong \mathscr{T}'$?
		\item \label{que2}\hspace{-.15cm}. What are the maximal isomorphic classes of $E(K)_{\tors}$ of an elliptic curve~$E$ over $K$ (as an analogue of Corollary~\ref{cor:maxQ})?
		\item \label{que3}\hspace{-.15cm}. Assume that there exists an elliptic curve  over $K$ with a finite abelian group $\mathscr{T}$ as its torsion subgroup over $K$. Are there infinitely many (non-isomorphic) elliptic curves~$E$ defined over $K$ such that $E(K)_{\tors} \cong \mathscr{T}$?
		\item \label{que4}\hspace{-.15cm}.  Assume that there exists an elliptic curve  over $K$ with a finite abelian group $\mathscr{T}$ as its torsion subgroup over $K$. Then, do almost all elliptic curves~$E$ over $K$ such that $E(K)_{\tors} \supseteq\mathscr{T}$ have its torsion subgroup $\mathscr{T}$ over $K$ exactly, i.e,. $E(K)_{\tors} \cong\mathscr{T}$? 
	\end{enumerate}
\end{quest*}

The questions \ref{que1}--\ref{que4} are still open for almost all number fields, even for quadratic extensions of $\bbq$. Note that if the answer to \ref{que4} is positive, then so is the answer to \ref{que1}. Moreover, if \ref{que4} has an affirmative answer, then we may consider only the maximal isomorphic classes among the torsion subgroups of elliptic curves defined over $K$ for \ref{que2} and~\ref{que3}.

We will introduce some partial answers to	the questions~\ref{que1}--\ref{que3} over the quadratic extension $K=\bbq(\sqrt{-1})$:
Kenku and Momose \cite{Kenku_Momose} have provided a list of isomorphic classes of groups to which every realizable torsion subgroup of an elliptic curve defined over $\bbq(\sqrt{-1})$  necessarily belongs, and  Najman \cite{Najman} has refined the list by finding some classes which are not realizable over  $\bbq(\sqrt{-1})$. 
Then, Zhao \cite[Theorem~2.3]{Zhao} has computed $d_{\mathscr{T}}$ in Theorem~\ref{thm:old_Harron_Snowden} over $\bbq(\sqrt{-1})$ for each torsion subgroup $\mathscr{T}$ belonging to Najman's refined list such that  $\mathscr{T}\ne 0, \bbz/2\bbz, \bbz/3\bbz$. 
We note that \cite[Theorem~2.3]{Zhao} shows that $d_{\bbz/13\bbz}=\infty$, 
and Najman \cite{Najman2} has proved that $\bbz/13\bbz$ would not be realizable as a torsion subgroup of any elliptic curve over $\bbq(\sqrt{-1})$. Moreover, Zhao \cite{Zhao} has given a partial answer to \ref{que4} over the quadratic extension~$\bbq(\sqrt{-1})$ by proving
that $d_{\mathscr{T}}$ increases strictly as $\mathscr{T}$ gets bigger.

Also, Trbovi{\` c} \cite{Antonela} has provided a list of possible isomorphic classes of the torsion groups of elliptic curves over  quadratic fields $\bbq(\sqrt{D})$ for square-free integers $2\le D<100$, but without specifying which groups are  actually realizable.

On the other hand, there are several partial answers to the  questions~\ref{que1}--\ref{que3} in terms of small degrees of a given number field $K$ over $\bbq$. In \cite{Gu21}, Gu{\v z}vi{\` c}	has given a good summary of the related results under various conditions. In particular, if we let $d=[K:\bbq]$, then  for~$d=2$, some weakened answers to \ref{que1}--\ref{que3} are given. In particular,  
Kenku-Momose  \cite{Kenku_Momose} and Kamienny \cite{Kamienny} have given possible isomorphic classes of the torsion subgroup of an elliptic curve defined over a quadratic number field.
For $3\leq d\leq 6$, the lists of all possible and realizable isomorphic classes $\mathscr{T}$ of groups such that there are infinitely many elliptic curves $E$ over some number fields $K$ of degree $d$ with $E(K)_{\tors}\cong \mathscr{T}$ are given \cite[Theorem~3.4]{J3} for $d=3$, \cite[Theorem~3.6]{J4} for $d=4$, and \cite[Theorem~1.1]{DS56} for $d=5, 6$. Moreover, they have proved that for each $d=3,4,5,6$, if an abelian group $\mathscr{T}$ is in their list, then so are the subgroups of $\mathscr{T}$. 

As a generalization of Theorem~\ref{thm:old_Harron_Snowden} to  an arbitrary number field $K$, Bruin and Najman~\cite{Bruin_Najman} considered elliptic curves with all level structures of the torsion subgroup~$G$ such that the corresponding modular curves over $K$ satisfy certain conditions, and counted how many such elliptic curves have the $G$-level structure  for each subgroup $G \subseteq \GL{N}$. More precisely, under the certain assumptions on the weights $\omega$ of the moduli stacks of such elliptic curves and their `reduced degrees' $e(G)$, they proved that there exists a constant $\delta_{G,K}$ depending on~$\omega$ and $e(G)$ such that \begin{equation} \label{BN_delta}
	 \lim_{X\to \infty}\frac{\log\#\{\left(A,B\right)\in\scrU(K):\cH\left(A,B\right)\le X, E_{A,B} \text{ admits $G$-level structure}\}}{\log X}=\dfrac{1}{\delta_{G,K}}.
 \end{equation}
	
Note that since this constant $\delta_{G,K}$ depends on the weights and the reduced degrees, it is not clear whether $\delta_{G,K}$ increases or not as $G$ gets bigger, so the answer to \ref{que4} over all number fields cannot be deduced directly by this result. We give more explanation about this in Remark~\ref{remark:oldresult}.

On the other hand, regarding \ref{que4}, if $\mathscr{T}$ is trivial, then Duke  gave an answer when  $K=\bbq$, and then Zywina  generalized Duke's result to a number field $K$ such that $K \cap \bbq^{\operatorname{cyc}} = \bbq$, where $\bbq^{\operatorname{cyc}}$ is the cyclotomic extension of $\bbq$ as follows:

\begin{theorem}[{\cite[Corollary~1.2]{Duke}, \cite[Theorem~1.2]{Zywina}}] \label{thm:duke}  
	Let $K$ be a number field such that $K \cap \bbq^{\operatorname{cyc}} = \bbq$. Then, for almost all elliptic curves $E$ defined over $K$, $E(K)_{\tors}$ is trivial.
\end{theorem}

In this paper, we generalize Theorem~\ref{thm:duke}  for certain torsion subgroups over all number fields and give an answer to \ref{que4} in these cases. Before stating our main theorem, we give a brief introduction to a modular curve $X_{1}(m,n)$. The reason for introducing $X_{1}\left(m,n\right)$ is that our main result depends on its genus.

Let $\mathbb{H}^{*}=\{z\in \bbc: \Im z>0\}\cup \bbq \cup \{\infty\}$. For positive integers $m$ and $n$ such that $m\mid n$, the modular curve $X_{1}(m,n)$ is defined by  the quotient  $X_{1}(m,n)=\Gamma_{1}(m,n)\backslash\mathbb{H}^{*}$ of $\mathbb{H}^{*}$ by the action of the  congruence subgroup,
\begin{align*}
	\Gamma_{1}&(m,n)\\
	&:=\left\{\begin{pmatrix}a&b\\c&d\end{pmatrix}\in \operatorname{SL}_{2}(\bbz): a\equiv 1,  b\equiv 0 \pmod {m\bbz},\text{ and }
	c\equiv 0,d\equiv 1 \pmod {n\bbz}\right\}.
\end{align*}
We note that $X_{1}(m,n)$ is an algebraic curve defined over $\bbq(\zeta_{m})$ where $\zeta_{m}$ is a primitive $m$th roots of unity (see \cite[V.4.4]{DR}). We denote by $g_{m,n}$ the genus of $X_{1}(m,n)$.

We can refine the questions \ref{que3} and \ref{que4} a bit more depending on the genus $g_{m,n}$. If $g_{m,n}\ge 2$, Faltings' theorem (\cite[Satz~7]{Faltings}) implies that $X_{1}(m,n)(K)$ is finite for a number field~$K$. containing  $\bbq(\zeta_{m})$. Hence, \ref{que3} and \ref{que4} make sense only for a finite abelian group $\mathscr{T} \cong \bbz/m\bbz \times \bbz/n\bbz$ such that $g_{m,n}\le 1$. If $g_{m,n}=1$, then	$X_{1}(m,n)(K)$ can be finite or infinite depending on the given number field $K$ containing  $\bbq(\zeta_{m})$. If $g_{m,n}=0$, then $X_{1}(m,n)(K)$ is infinite if and only if it is not empty.  Moreover, the previous known results also reflect this phenomenon. For example, $\mathscr{T}\cong \bbz/m\bbz \times \bbz/n\bbz$ can be realized as a torsion subgroup of an elliptic curve over~$\bbq$  and $\bbq(\sqrt{-1})$ (see \cite{Mazur}, \cite{Najman} and \cite{Najman2}) only if $g_{m,n}=0$, considering the list of all $(m,n)$ such that $g_{m,n}=0$ given in \cite[Theorem~0.1]{J5}, and for such a group $\mathscr{T}$, the answer for \ref{que4} is positive for the number fields $\bbq$ and $\bbq(\sqrt{-1})$ (Corollary~\ref{cor:Harron_Snowden},\cite{Zhao}).

Therefore, our interest of this paper is focused on the cases of the torsion subgroups isomorphic to $\bbz/m\bbz\times \bbz/n\bbz$ when  the genus $g_{m,n}$ is $0$. 
If we let  \begin{equation}\label{eqn:genus0}
	T_{g=0} :=\left\{(m,n)\in \bbz^{>0}\times \bbz^{>0} : m\mid n, \text{ and } g_{m,n}=0\right\}, 
\end{equation} then by \cite[Theorem~0.1]{J5}, we can conclude that \begin{equation*}
	T_{g=0} =  \left\{(1,n_{1}),(2,n_{2}),(3,3),(3,6),(4,4),(5,5):n_1\in\{1,2,\ldots, 10, 12\}, n_{2}\in\{2,4,6,8\}\right\}.
\end{equation*}

Finally, we are ready to state	our main theorem:

\begin{theorem}\label{thm:main}
	Let $K$ be a number field. For each $(m,n)\in T_{g=0}$, if there exists an elliptic curve over $K$ whose torsion subgroup over $K$ contains a subgroup $\mathscr{T}  \cong\bbz/m\bbz\times \bbz/n\bbz$, then almost all elliptic curves $E$ defined over $K$  such that $E(K)_{\tors}\supseteq\mathscr{T}$ satisfy that $E(K)_{\tors}\cong\mathscr{T}$ exactly, i.e., $$
		\lim_{X\to \infty}\frac{\#\{\left(A,B\right)\in \scrU(K):\cH\left(A,B\right)\le X, {E_{A,B}(K)}_{\tors} \cong \mathscr{T}\}}{\#\{\left(A,B\right)\in \scrU(K):\cH\left(A,B\right)\le X,{E_{A,B}(K)}_{\tors} \supseteq \mathscr{T}\}}=1.
	$$
\end{theorem}

The following is an application of Theorem~\ref{thm:main}.

\begin{cor}\label{cor:easy_cor}
	Let $K$ be a number field. For each $(m,n)\in T_{g=0}$, there is an elliptic curve~$E_{A,B}$ over $K$ such that ${E_{A,B}(K)}_{\tors} \cong \bbz/ m \bbz \times \bbz / n \bbz$ if and only if $\zeta_{m}\in K$, where~$\zeta_{m}$ is a primitive $m$th root of unity.
\end{cor}

The main strategy to prove  Theorem~\ref{thm:main}, in particular, to prove that the desired torsion subgroup property in Theorem~\ref{thm:main} holds for `almost all' elliptic curves over a given number field $K$,  is to apply the Hilbert irreducibility theorem (HIT) (see Theorem~\ref{thm:HIT}) to some parameterized polynomials defining $N$-torsion points of an elliptic curve and its specializations.

If we describe it in more detail,
let $V$ be an absolutely irreducible variety defined over $K$ and $f(X)=X^{n}+\sum_{1\le i\le n}a_{i}X^{n-i}\in K(V)[X]$ be a monic polynomial over the function field $K(V)$ of $V$. For each $t\in V(K)$, let $f_t$ be the specialization of $f$ at $t$. Then,  the Galois groups $\mathcal{G}_{t}$ of $f_t$ over $K$ is a subgroup of the Galois group $\mathcal{G}$ of $f$ over $K(V)$ in general, and  HIT (Theorem~\ref{thm:HIT}) states that the Galois groups $\mathcal{G}_{t}$  is exactly the Galois group $\mathcal{G}$ for each $t\in V(K)$ outside a ``thin set'' (see Definition~\ref{defn:thinset} for its definition). 
On the other hand, \ref{que4} asks whether torsion subgroups  of almost all elliptic curves whose torsion subgroups contain a prescribed torsion subgroup $\mathscr{T}$ are exactly  $\mathscr{T}$ (i.e., no larger than $\mathscr{T}$),
and HIT (Theorem~\ref{thm:HIT}) can be applied to give a partial answer to \ref{que4} by regarding	`a thin set' in terms of  the Galois groups corresponding to the complement of a set of elliptic curves with the desired torsion subgroup property given in Theorem~\ref{thm:main}.

\begin{remark}\label{remark:oldresult}
	Most known results mentioned in the introduction (e.g., \cite{Duke}, \cite{Najman}, and \cite{Zhao}) hold only over $\bbq$ or certain number fields of bounded degrees. Moreover, they require  knowledge all classes of realizable torsion subgroups $\mathscr{T}$ and explicit computations of both quantities $d_\mathscr{T}$ and  $d_\mathscr{T'}$ (the reciprocals of densities as in Theorem~\ref{thm:old_Harron_Snowden})  of pairs of realizable torsion subgroups $\mathscr{T}\subsetneq \mathscr{T'}$. This is necessary to determine whether  $d_\mathscr{T}$ increases as $\mathscr{T}$ grows and consequently, to answer~\ref{que4}.
	However, the complete classification of realizable torsion subgroups~$\mathscr{T}$ over general number fields~$K$ remains unknown.

	In contrast, our approach does not require full knowledge of all realizable torsion subgroups. Instead, we focus on determining whether a finite abelian group $\mathscr{T}$ can be realized as a subgroup of the torsion subgroup of an elliptic curve over $K$, without needing to compute densities for all possible extensions $\mathscr{T}\subsetneq \mathscr{T}'$. This is made possible by employing a more Galois-theoretic approach. Regarding the method of Bruin and Najman \cite{Bruin_Najman}, they computed $\delta_{G,K}$ (as defined in \eqref{BN_delta}) when $g_{m,n}=0$, where $G= \Gamma_1(m,n)$ is viewed as a subgroup of $\GL{n}$. However, when $g_{m,n}=1$, even if an elliptic curve $E$ over $K$ exists such that ${E\left(K\right)}_{\tors}$ contains $\bbz/m\bbz\times \bbz/n\bbz$, their method may not allow for a computation of $\delta_{G,K}$ or a comparison of densities for such elliptic curves over $K$. Therefore, their approach does not directly lead to our main result, Theorem~\ref{thm:main}, whereas our method remains applicable in these cases.

	In our paper, in order to give an answer to \ref{que4} over arbitrary number fields $K$,  we only need to know the realization of a finite abelian group $\mathscr{T}$ as {\bf a subgroup of the torsion subgroup} of an elliptic curve over $K$, but there is no need to know either the realization of each $\mathscr{T'}$ and each $\mathscr{T}$ such that $\mathscr{T}\subsetneq \mathscr{T'}$, or compute the density of each of them, by applying more Galois group theoretic approach.

\end{remark}

The structure of our paper is as follows: 
In Section~\ref{sec:HIT}, we define a thin set and give the proof of the Hilbert irreducibility theorem described above.

In Section~\ref{sec:trans}, we interpret and transform the $N$-torsion subgroup problem into the Galois group problem of parameterized polynomials associated with the division polynomials (Definition~\ref{defn:primitive_division_polynomial}) of elliptic curves. We give several equivalent statements for the torsion subgroup of an elliptic curve to contain a prescribed subgroup, and sufficient conditions for it to be exactly the given subgroup in terms of the Galois group conditions, which is the conclusion in Theorem~\ref{thm:min_Gal}.

In Section~\ref{sec:main}, we  parameterize all elliptic curves over a given number field $K$ satisfying a necessary condition (Condition~$\cP(m,n)$ in \eqref{eqn:condP} of Section~\ref{sec:trans}) to have their torsion subgroups contain $\bbz/m\bbz \times \bbz/n\bbz$ for each $(m,n)\in T_{g=0}$ in \eqref{eqn:genus0} and prove that these parameterizations satisfy the Galois group conditions of Theorem~\ref{thm:min_Gal}. The genus $0$ condition ($g_{m,n}=0$) of $X_1(m,n)$ is crucial to  parameterize such elliptic curves, so we complete the proofs of Theorem~\ref{thm:main} and Corollary~\ref{cor:easy_cor}.

\section{Thin sets and Hilbert's Irreducibility Theorem}\label{sec:HIT}

First, we define a thin set. We refer to \cite[\S3.1]{Serre08} for more details.
\begin{defn}\label{defn:thinset}
		For an irreducible $K$-variety $V$, we define the family $\mathscr{F}$ of thin sets in $V(K)$ by the minimal family satisfying the following conditions:
	\begin{enumerate}[\normalfont(a)]
		\item For a $K$-morphism $\phi:W\to V$ where  $W$ is an irreducible $K$-variety such that $\dim W\le \dim V$, $\phi(W(K))\in \mathscr{F}$  is called {\it thin}, if one of the following statements is satisfied:
		\begin{enumerate}[\normalfont(i)]
			\item $\dim W<\dim V$.
			\item $\phi$ is geometrically surjective of $\deg \phi>1$.
		\end{enumerate}
		\item If $Z_1,Z_2\in\mathscr{F}$, then $Z_1\bigcup Z_2\in \mathscr{F}$.
		\item If $Z_1\in\mathscr{F}$ and  $Z_{2}$ is a subset of $Z_{1}$, then $Z_2\in \mathscr{F}$.
	\end{enumerate}
\end{defn}
In order to prove Theorem~\ref{thm:HIT}, our main result of this section, known as  the Hilbert irreducibility theorem, we need the following fact.

\begin{prop}\label{prop:thin_set}
	Let $\mathcal{S}\subseteq \bbp^{n}(K)$ be a thin set. Then, for almost all points $P\in \bbp^{n}(K)$, $P\notin \mathcal{S}$.
\end{prop}
\begin{proof} Referring to Definition~\ref{defn:almostall}, we show that 
	$$
		\lim_{X\to \infty}\frac{N_{\mathcal{S}}(X)}{N_{\bbp^n(K)}(X)}=0.
	$$
	This follows from \cite[\S13.1 Theorem~3]{Serre97} and \cite[\S2.5 Theorem(Schanuel)]{Serre97} which prove that for any $0<\varepsilon$, 
	$N_{\mathcal{S}}(X) = O\left(X^{(n+1/2)d+\varepsilon}\right)$ and $N_{\bbp^{n}(K)}(X) \approx c_{n,K}X^{(n+1)d}$ where $d=[K:\bbq]$ and $c_{n,K}$ is a positive constant depending on $n$ and $K$.
\end{proof}


Let $V$ be an absolutely irreducible variety defined over $K$ and $f(X)=X^{n}+\sum_{1\le i\le n}a_{i}X^{n-i}\in K(V)[X]$ be a monic polynomial over the function field $K(V)$ of $V$. We define the specialization $f_{t}$ of $f$ at $t\in V(K)$ by $f_{t}(X)=X^{n}+\sum_{1\le i\le n}a_{i}(t)X^{n-i}\in K(V)[X]$. We denote by~$\mathcal{G}$ and $\mathcal{G}_{t}$ the Galois groups of $f$ over $K(V)$ and of $f_{t}$ over $K$, respectively. In general, we have the inclusion $\mathcal{G}_{t}\subseteq \mathcal{G}$. Let $V_{f}:=\{(t,x)\in V\times\bba^{1}:f_{t}(x)=0\}$ and let $W$ be the Galois closure of the covering $V_{f} \to V$  defined by $(t,x)\mapsto t$. Then, the Galois group of the cover $W\to V$ is~$\mathcal{G}$.

\begin{prop}[{\cite[Proposition~3.3.5]{Serre08}}]\label{prop:old_HIT}
	Let  $V/K$ be an absolutely irreducible smooth variety defined over a number field $K$ with $\dim(V) \ge 1$ and $f$ be a monic polynomial in  $ (K[V])[x]$ where $K[V]$ is coordinate ring of $V$. Then, there is a thin set $S\subseteq V(K)$ such that any point $t\in V(K)-S$ satisfies that $\mathcal{G}_{t}=\mathcal{G}$.
\end{prop}

\begin{remark}
	In fact, \cite[Proposition~3.3.5]{Serre08} assumes the irreducibility of $f$ which leads to the irreducibility of the specializations $f_t$, but to prove our main theorem, we need neither the irreducibility of $f$  nor that of $f_t$ in Proposition~\ref{prop:old_HIT}. For more details, refer to \cite[§3.3]{Serre08} and \cite[Proposition~3.3.1]{Serre08} whose corollary is  \cite[Proposition~3.3.5]{Serre08}.
\end{remark}

Now, we are ready to prove the Hilbert irreducibility theorem (HIT) that we need. 

\begin{theorem}\label{thm:HIT}
	Let $V$ be a $K$-rational variety  of $\dim \ge 1$ and $f\in (K[V])[X]$ a monic polynomial  over the coordinate ring $K[V]$ of $V$. Then, for almost all $t\in V(K)$, we have that  $$\mathcal{G}_{t}=\mathcal{G}.$$
\end{theorem}
\begin{proof}
	This follows from Proposition~\ref{prop:thin_set} and Proposition~\ref{prop:old_HIT}.
\end{proof}

To apply Theorem~\ref{thm:HIT} to our main theorem, we need to reformulate the torsion subgroup problem in terms of the Galois group problem of polynomials associated with torsion points, which we present in the following section.

\section{$N$-Torsion subgroups and their associated Galois groups}\label{sec:trans}

For an elliptic curve $E/F$ defined over a field $F$ with characteristic $0$ and an integer $N\ge 1$, we denote by  ${E}(F)[N]$ and $E[N]$ the groups of $N$-torsion points defined over $F$ and over an algebraic closure  $\overline{F}$ of $F$, respectively. Recalling that ${E}(F)[N]\subseteq {E}[N]\cong \bbz/N\bbz\times \bbz/N\bbz$, we consider them as subgroups of $\bbz/N\bbz\times \bbz/N\bbz$. 

\begin{lemma}\label{lem:abgp} For any positive integer $N$, suppose $ \mathscr{A}$ is  an abelian group such that $ \mathscr{A}\cong \bbz/N\bbz\times \bbz/N\bbz$ and $\mathscr{B}$  is a subgroup of $\mathscr{A}$.  Then there exist positive divisors  $m$ and $n$ of $N$  such that $m \mid n $	and $\mathscr{B}\cong \bbz/m\bbz\times \bbz/n\bbz$. In particular, there exist $x, y\in \mathscr{A}$ such that  $\mathscr{A}=\langle x, y \rangle$  and $\mathscr{B}=\left\langle \frac{N}{m}x,\frac{N}{n}y\right\rangle$.
\end{lemma}
\begin{proof}
	We may assume that $\mathscr{B}\cong \bbz/m\bbz\times \bbz/n\bbz\subseteq \bbz/N\bbz\times \bbz/N\bbz$ for some positive divisors $m$ and $n$ of $N$ such that $m\mid n$. Then there exist $h,k\in \mathscr{B}$ with $|h|=m$ and $|k|=n$ such that $\mathscr{B}=\langle h,k\rangle$ and 
	$  \langle h \rangle\cap \langle k\rangle=\{0\}$, where $\left|g\right|$ denotes the order of  $g \in \mathscr{A}$. We can find  $y\in \mathscr{A}$ such that  $|y|=N$ and $k=\frac{N}{n}y$ as follows:  We let $\mathscr{C}:=\{\alpha\in \mathscr{A}: n\alpha=0\}\cong \bbz/n\bbz\times \bbz/n\bbz$,
	$$X:=\{\alpha\in \mathscr{A}:|\alpha|=N\}\text{, and } Y:=\{\alpha\in \mathscr{A}:|\alpha|=n\}=\{\alpha\in \mathscr{C}:|\alpha|=n\}.$$
	If we consider the function $f:X\to Y$ given by $\alpha\mapsto \frac{N}{n}\alpha$, then $f$ commutes with any  $\sigma \in \operatorname{Aut}(\mathscr{A})$, i.e., $f(\sigma(\alpha))=\sigma(f(\alpha))$ for all $\sigma \in \operatorname{Aut}(\mathscr{A})$ and all $\alpha\in X.$ Hence, ${\sigma(f^{-1}(\gamma))}=f^{-1}(\sigma(\gamma))$ for any  $\sigma \in \operatorname{Aut}(\mathscr{A})$ and $\gamma\in Y$. Since $\operatorname{Aut}(\mathscr{C})\cong \GL{n}$ acts on $Y$ transitively and the natural homomorphism $\operatorname{Aut}(\mathscr{A})\cong \GL{N}\to \GL{n}\cong \operatorname{Aut}(\mathscr{C})$ is surjective, $\operatorname{Aut}(\mathscr{A})$ also acts on $Y$ transitively. Therefore, $f^{-1}(\gamma)$ is not empty for any $\gamma\in \mathscr{A}$ such that $\mid \gamma \mid =n$. Hence there is $y\in \mathscr{A}$ of order $N$ such that $\frac{N}{n}y=k$. 
	Next, we choose $x\in\mathscr{A}$ such that $\mathscr{A}=\langle x,y \rangle$. Then $h=\frac{N}{m}(ax+by)$ for some integers $a$ and $b$.
	Then, we can show $\gcd(a,m)=1$. In fact, if we let  $d:=\gcd(a,m)$, then since  $d\mid m$, $m\mid n$, and $\left| x \right| = N$, $$
		\frac{m}{d}h=\frac{a}{d} (Nx)+\frac{m}{d}\left(\frac{N}{m}by\right)=\frac{n}{d}b\left(\frac{N}{n}y\right)\in \langle h\rangle \cap \langle k\rangle=\{0\},
	$$ so $m\mid \frac{m}{d}$, which implies that $d=1$.
	Therefore, there are two integers $s$ and $t$ such that $at+ms=1$. Then, $\frac{N}{m}x=\frac{N}{m}(at+ms)x=\frac{N}{m}atx\in \langle \frac{N}{m}ax\rangle\subseteq \langle \frac{N}{m}x\rangle$ and so $$
		\mathscr{B} =\left\langle \frac{N}{m}(ax+by), \frac{N}{n}y\right\rangle =\left\langle \frac{N}{m}ax, \frac{N}{n}y\right\rangle=\left\langle \frac{N}{m}x, \frac{N}{n}y\right\rangle.
	$$
\end{proof}

For any integer $N\ge 1$ and $\left(A,B\right)\in \scrU(K)$, the field $K(E_{A,B}[N])$, which is the field of definition of $E_{A,B}[N]$ over $K$, is Galois over $K$.
We denote its Galois group by $$ G_{N,A,B}:=\Gal(K(E_{A,B}[N])/K).$$ Then for an ordered basis $\cB=\{P,Q\}$ of $E_{A,B}[N]\cong\bbz/N\bbz\times \bbz/N\bbz$,  we have `the right action' of $G_{N,A,B}$ on  $E_{A,B}[N]$, i.e., if $\sigma, \tau \in G_{N,A,B}$, $P^\sigma = a_{\sigma} P+b_{\sigma} Q$ and $Q^\sigma = c_{\sigma} P+d_{\sigma} Q$, for some $a_{\sigma},b_{\sigma},c_{\sigma},d_{\sigma}\in \bbz/N\bbz$, and $(R^\sigma)^\tau=R^{\sigma\tau}$, for $R\in E_{A,B}[N]$. Hence,  by identifying $\GL{N}$ with $\text{Aut}(E_{A,B}[N])$, and identifying each $\sigma \in G_{N,A,B}$  with  $\begin{pmatrix}a_{\sigma}&b_{\sigma}\\c_{\sigma}&d_{\sigma}\end{pmatrix}\in \GL{N}$,  we define an injective group homomorphism $r_{\cB}:G_{N,A,B}\to \GL{N}$ under the right action by 
$$
	r_{\cB}\left(\sigma\right)=	\begin{pmatrix}a_{\sigma}&b_{\sigma}\\c_{\sigma}&d_{\sigma}\end{pmatrix}, \quad \text{ where }
	\begin {pmatrix} P^\sigma\\Q^\sigma\end{pmatrix} =\begin{pmatrix}a_{\sigma}&b_{\sigma}\\c_{\sigma}&d_{\sigma}\end{pmatrix} \begin{pmatrix}P\\Q\end{pmatrix}.
$$

\begin{lemma}\label{lem:demo}
	Let $N\geq 1$ be an integer, and  $m$ and $n$ be positive divisors  of $N$ such that $m \mid n$. Then, for $\left(A,B\right)\in \scrU(K)$, the following  are equivalent:
	\begin{enumerate}[\normalfont(a)]
		\item $ E_{A,B}(K)[N]$ contains $\bbz/m\bbz\times \bbz/n\bbz$.
		\item There is an (ordered)  basis $\cB$ of $E_{A,B}[N]$ such that $r_{\cB}\left(G_{N,A,B}\right)\subseteq H_{N}(m,n)$,
		where \begin{align*}
			H_{N}&(m,n)\\
			&:=\left\{\begin{pmatrix}a&b\\c&d\end{pmatrix}\in \GL{N}: a\equiv 1,  b\equiv 0 \pmod {m\bbz},\text{ and }
			c\equiv 0,d\equiv 1 \pmod {n\bbz}\right\}.
		\end{align*}
	\end{enumerate} 

\end{lemma}
\begin{proof}
	If $\bbz/m\bbz\times \bbz/n\bbz\subseteq E_{A,B}(K)[N]$, then by Lemma~\ref{lem:abgp}, there is a basis $\cB=\{P,Q\}$ of $E_{A,B}[N]$ such that $\tfrac{N}{m}P,\tfrac{N}{n}Q\in E_{A,B}(K)[N]$. The condition that $\tfrac{N}{m}P$ and $\tfrac{N}{n}Q$ are fixed under $G_{N,A,B}$  implies the congruence conditions for matrix entries in ~$H_N(m,n)$.
	
	The converse can be proved similarly.
\end{proof}

Next, we recall the definition of the division polynomials of an elliptic curve.

\begin{defn}[the  division polynomials, {\cite[III.Exercises~3.7]{Silverman}}]\label{defn:division_polynomial}
For a non-negative integer~$n$,	 the $n$th division polynomial $\psi_{n}$ of an elliptic curve $E_{s,t} : y^2=x^3+sx+t$ over $K(s,t)=K(\scrU)$  is defined inductively by
\begin{align*}
	&\psi_{0}=0,\qquad \psi_{1}=1,\qquad \psi_{2}=2y,\\
	&\psi_{3}=3x^{4}+6sx^{2}+12tx-s^{2}, \quad 	\psi_{4}=4y(x^{6}+5sx^{4}+20tx^{3}-5s^{2}x^{2}-4stx-(8t^{2}+s^{3})),\\
	&\psi_{2k-1}=\psi_{k+1}\psi_{k-1}^{3}-\psi_{k-2}\psi_{k}^{3}, \quad \text{ and }\quad  \psi_{2k}=\frac{\psi_{k}}{2y}\left(\psi_{k+2}\psi_{k-1}^{2}-\psi_{k-2}\psi_{k+1}^{2}\right),\quad \text{ for }k\ge 3.
	\end{align*}
\end{defn}

For each  integer $n\geq 1$, let \begin{align}\label{eqn:phi-omega}
	\phi_{n}=x\psi_{n}^{2}-\psi_{n+1}\psi_{n-1}, \quad \text{ and } \quad \omega_{n}=\frac{\psi_{n+2}\psi_{n-1}^{2}-\psi_{n-2}\psi_{n+1}^{2}}{4y}. 
\end{align} Then, for any point $P=(x,y)\in E_{s,t}$, $nP=\left(\dfrac{\phi_{n}(P)}{(\psi_{n}(P))^{2}},\dfrac{\omega_{n}(P)}{(\psi_{n}(P))^{3}}\right)$. 

In particular, $x\left(nP\right)=\dfrac{\phi_{n}(P)}{(\psi_{n}(P))^{2}}=\dfrac{\phi_{n}(x)}{(\psi_{n}(x))^{2}}$.

For any integer $N\geq 1$, we recall the following (for example, see \cite[III.Exercises~3.7]{Silverman}):

\begin{itemize}
	\item If $N$ is odd, then  $\frac{\psi_{N}}{N}\in (\bbq\left[s,t\right])[x]\subsetneq (K(\scrU))[x]$ is monic of  $\deg_{x}\frac{\psi_{N}}{N}=\frac{N^{2}-1}{2}$, and if $N\geq 3$ is odd, the zeros of $\frac{\psi_{N}}{N}$ are the $x$-coordinates of points in $E_{s,t}[N]-\{O\}$.
	\item If $N$ is even, then $\frac{\psi_{N}}{Ny}\in (\bbq\left[s,t\right])[x]\subsetneq (K(\scrU))[x]$ is monic of  $\deg_{x}\psi_{N}=\frac{N^{2}-4}{2}$, and if $N\geq 4$ is even, the zeros of $\frac{\psi_{N}}{Ny}$ are the $x$-coordinates of points in $E_{s,t}[N]-E_{s,t}[2]$.
\end{itemize}

Then, we construct the following polynomial whose solutions are exactly the $x$-coordinates of the points of order $N$ by using the M{\" o}bius function $\mu$.

\begin{defn}[the primitive division polynomials and their Galois groups]\label{defn:primitive_division_polynomial}
	For each integer $N\ge 1$, we define the~$N$th primitive division polynomial $\Psi_{N}$ as follows; 
	$$\Psi_{1}=1, \quad \text{ and } \quad \Psi_{2}=x^{3}+sx+t \in \bbq\left[s,t\right][x]\subseteq K(\scrU)[x],$$
	$$
		\Psi_{N}:=\prod_{d\mid N}\left(\frac{\psi_{d}}{d}\right)^{\mu(N/d)}\in \bbq\left[s,t\right][x]\subseteq K(\scrU)[x], \text{ for } N\geq 3,
	$$ where $\mu$ is the M{\" o}bius function.
	
	For each integer $N\ge1$, we denote by $\Gbar_{N}$ the Galois group of $\Psi_{N}$ for $E_{s,t}$ over $K(s,t)$.
	Let the polynomial $\Psi_{N,A,B}$ over $K$ be the specialization of $\Psi_{N}$ at $\left(A,B\right)\in \scrU(K)$. 
\end{defn}

First, we give some properties of the primitive division polynomials in Lemma~\ref{lem:psi_deg} and Remark~\ref{rmk:psi_N} below. For a positive integer $N$, we let \begin{equation}\label{eqn:delta}
	\delta_{N} := N^{2}\prod\limits_{\text{primes }p\mid N}(1-\frac{1}{p^{2}}),
\end{equation} as a more convenient expression to simplify later use.
\begin{lemma}\label{lem:psi_deg}
	For an integer $N\ge 3$, we have that $\deg_{x}\Psi_{N} = \frac{1}{2}\delta_{N}$, which is the half of the number of torsion points in $E_{s,t}(\overline{K(\scrU)})$ of order $N$.
\end{lemma}
\begin{proof}
	The definition of $\Psi_{N}$ implies that $\deg_{x}\Psi_{N}$ is equal to the half of the number of torsion points in $E_{s,t}(\overline{K(\scrU)})$ of exact order $N$. Since $E_{s,t}[N]\cong (\bbz/N\bbz)^{2}$ as abelian groups, it is enough to show that $\delta_{N}=\#X$, where $\#X$ denotes the number of elements of $X:=\{\left(a,b\right)\in (\bbz/N\bbz)^{2}:\left|\left(a,b\right)\right|=N\}$. By the Chinese Remainder Theorem, we have $\#X=\prod\limits_{\text{primes }p\mid N}\#X_{p},$ where $X_{p}:=\{\left(a,b\right)\in (\bbz/p^{n_{p}}\bbz)^{2}:\left|\left(a,b\right)\right|=p^{n_{p}}\}$ for a prime factorization $N = \prod\limits_{\text{primes }p\mid N} p^{n_{p}}$. The principle of inclusion–exclusion yields $$
		\#X_{p}=2\cdot p^{n_{p}}\cdot (p^{n_{p}}-p^{n_{p}-1}) - (p^{n_{p}}-p^{n_{p}-1})^{2} = p^{2n_{p}}(1-p^{-2}),
	$$ where the first term counts the number of pairs $\left(a,b\right)\in (\bbz/p^{n_{p}}\bbz)^{2}$ in which  at least one of $a$ and $b$ is relatively prime to $p$, while the  second term represents the number of pairs $\left(a,b\right)\in (\bbz/p^{n_{p}}\bbz)^{2}$ where both~$a$ and $b$ are relatively prime to $p$.
\end{proof}

\begin{remark}\label{rmk:psi_N}
	For each integer $N\ge 1$, we note that:  \begin{enumerate}[\normalfont(a)]
		\item $\Psi_{N}\in (\bbq\left[s,t\right])[x]\subseteq (K(\scrU))[x]$, and $\Psi_{N}$ is monic.
		\item If $N\ge 2$, for $\left(A,B\right)\in \scrU(K)$, the zeros of $\Psi_{N,A,B}$ are the $x$-coordinates of torsion points in $E_{A,B}(\overline{K})$ of order $N$.
		\item $\Psi_{N}$ is separable over $K(\scrU)$.
		\item If we assign weights to the variables as $\mathrm{wt}(x)=1$, $\mathrm{wt}(s)=2$, and $\mathrm{wt}(t)=3$, then  $\Psi_{N}\in K(\scrU)[x]$ is a homogeneous polynomial with  $\mathrm{wt}(\Psi_N)=\deg_{x} \Psi_{N}$. In other words, $\Psi_{N,D^{2}s,D^{3}t}(Dx)=D^{\deg_{x} \Psi_{N}}\Psi_{N,s,t}(x)$ for $D$ in the multiplicative group $\Gm = \bba^{1} \setminus \left\{0\right\}$.
	\end{enumerate}
\end{remark}

Next, we note that if there exists $\nu\in G_{N,A,B}$  such that $P^{\nu}=-P $ for all $P\in E_{A,B}[N]$, then $\langle \nu\rangle \unlhd G_{N,A,B}$, so we define the group, $$
	\Gbar_{N,A,B}=\begin{cases}
	G_{N,A,B}/\left\langle \nu \right\rangle,& \text{ if there is } \nu\in G_{N,A,B}\text{ such that }P^{\nu}=-P \text{ for all }P\in E_{A,B}[N],\\
	G_{N,A,B,}& \text{ otherwise}.\\
	\end{cases}
$$
Then, it turns out that $\Gbar_{N,A,B}$ is the Galois group of $\Psi_{N,A,B}$ over $K$ as follows.

\begin{lemma}\label{lem:spltting_fld}
	For any $\left(A,B\right)\in \scrU(K)$ and each integer $N\ge 3$, the Galois group of $\Psi_{N,A,B}$ over $K$ is $\Gbar_{N,A,B}$. For a given basis $\cB$ of $E_{A,B}[N]$, we define the group homomorphism \begin{align}\label{bar-r}
	\rbar_{\cB}:\Gbar_{N,A,B}\to \GL{N}/\left\{ \pm I_{2}\right\} \text{ by } \overline{\sigma}\mapsto r_{\cB}\left(\sigma\right)\left\{ \pm I_{2}\right\},
	\end{align} for each $\sigma\in G_{N,A,B}$. Then, $\rbar_{\cB}$ is injective and $$
		\rbar_{\cB}(\Gbar_{N,A,B})=\left\langle -I_{2}, r_{\cB}\left(G_{N,A,B}\right)\right\rangle/\left\{ \pm I_{2}\right\}.
	$$
	
	In particular, if  there is no $\nu\in G_{N,A,B}$ such that $P^{\nu} = -P$, then $$
		\rbar_{\cB}(\Gbar_{N,A,B})\cong r_{\cB}\left(G_{N,A,B}\right).
	$$
\end{lemma}
\begin{proof}
	Let $P_{1}:=(\alpha_{1}, \beta_{1})$ and $P_{2}:=(\alpha_{2},\beta_{2})\in E_{A,B}[N]$ form a basis of $E_{A,B}[N]\cong \bbz/N\bbz\times \bbz/N\bbz$.
	If we let	$P_{3}=P_{1}+P_{2}:=(\alpha_{3},\beta_{3})$, then $P_3\in E_{A,B}[N]$. First, we show  that the splitting field $L$ of $\Psi_{N,A,B}$ over $K$ is $K(\alpha_{1},\alpha_{2},\beta_{1}\beta_{2})$. It is enough to show that $$
		\left\{ \sigma \in G_{N,A,B}: \gamma^{\sigma} = \gamma\text{ for all } \gamma \in L \right\}
		= \left\{ \sigma \in G_{N,A,B}: \gamma^{\sigma} = \gamma\text{ for all } \gamma \in K(\alpha_{1},\alpha_{2},\beta_{1}\beta_{2}) \right\}.
	$$  
	
	If  $\sigma \in G_{N,A,B}$ fixes $L$, then for each $i=1,2,3$, we have $\alpha_{i}^{\sigma}=\alpha_{i}$  and $\beta_{i}^{\sigma}=\beta_{i}$ or $-\beta_{i}$, i.e., $\alpha_{i}$ are fixed under $\sigma$, and  $P^{\sigma} = \pm P$ for any $P\in E_{A,B}[N]$. If we let $P_{i}^{\sigma} = \epsilon_{i} P_{i} $ for $\epsilon_{i} \in \left\{\pm 1\right\}$, then we conclude that $\epsilon_{1}=\epsilon_{2}=\epsilon_{3}$ due to the assumption that $N\ge 3$. 
	Hence, there is a sign~$\epsilon\in \left\{\pm 1\right\}$ such that 
	\begin{equation}\label{spltting_fld_arguement}
		P_{i}^{\sigma}=\epsilon P_{i}\text{ for all } i=1,2,3, \quad \text{i.e., }\beta_{1}\beta_{2}\text{ is fixed by }\sigma.
	\end{equation}

	Conversely, if $\sigma \in G_{N,A,B}$ fixes $K(\alpha_{1},\alpha_{2},\beta_{1}\beta_{2})$, then for any integers $m_1$ and $m_2$, the $x$-coordinate of  $m_{1} P_{1} + m_{2} P_{2}$ is fixed under $\sigma$. Hence $ L=K(\alpha_{1},\alpha_{2},\beta_{1}\beta_{2})$.

	If $\beta_{1}\notin L$ (so $\beta_2\notin L$), then there exists $\nu\in G_{N,A,B}$ such that $P^{\nu} = -P$ for all $P \in E_{A,B}\left[N\right]$, i.e., $r_{\cB}\left(\nu\right) = -I_{2}$, and $\Gal(L/K)= G_{N,A,B}\big / \left\langle \nu \right\rangle = \Gbar_{N,A,B}$ since $\Gal(K(P_{1},P_{2})/L))=\left\langle \nu \right\rangle$. Therefore, $\rbar_{\cB}$ is well-defined and $$
		\rbar_{\cB}(\Gbar_{N,A,B})= r_{\cB}\left(G_{N,A,B}\right)/\left\{ \pm I_{2}\right\}=\left\langle -I_{2}, r_{\cB}\left(G_{N,A,B}\right)\right\rangle/\left\{ \pm I_{2}\right\}.
	$$ We note that $\rbar_{\cB}$ is injective since $r_{\cB}$ is injective and $\rbar_{\cB}^{-1}\left(\left\{\pm I_{2}\right\}\right) = \left\langle \nu\right\rangle$.

	If $\beta_{1}\in L$ (so $\beta_2\in L$), then there is no such $\nu\in G_{N,A,B}$. So $\Gal(L/K)=G_{N,A,B}=\Gbar_{N,A,B}$ and $$
		\rbar_{\cB}(\Gbar_{N,A,B})= r_{\cB}\left(G_{N,A,B}\right)=\left\langle -I_{2}, r_{\cB}\left(G_{N,A,B}\right)\right\rangle/\left\{ \pm I_{2}\right\}.
	$$ We note that $\rbar_{\cB}$ is injective since $r_{\cB}$ is injective and $\ker \rbar_{\cB} = r_{\cB}^{-1}\left(\left\{\pm I_{2} \right\}\right)$ is trivial.
\end{proof}

We define the following groups which will be used throughout this paper.

\begin{defn}\label{defn:Gal} For positive divisors $m$ and $n$ of a positive integer $N$ such that $m\mid n$, we define three groups related to  the groups $H_{N}(m,n)$ defined in Lemma~\ref{lem:demo}:
\begin{enumerate}[\normalfont(a)]
	\item $\Hbar_{N}(m,n):=\left\langle H_{N}(m,n),-I_{2}\right\rangle\big / \{\pm I_{2}\}$
	\item $H^{1}_{N}(m,n):=H_{N}(m,n) \cap \SL{N}$
	\item $\Hbar^{1}_{N}(m,n):=\left\langle H^{1}_{N}(m,n), -I_{2}\right\rangle\big / \{\pm I_{2}\}$
\end{enumerate}
\end{defn}
The following lemma computes the index of $\Hbar^{1}_{N}(m,n)$ in $\Hbar^{1}_{N}(1,1)$.
\begin{lemma}\label{lem:gpidx}
	For positive divisors $m$ and $n$ of a positive integer $N$ such that $m \mid n$, we have that referring to $\delta_{n}$ given in \eqref{eqn:delta},$$
		\left[\Hbar^{1}_{N}(1,1):\Hbar^{1}_{N}(m,n)\right] = \left[\Hbar^{1}_{n}(1,1):\Hbar^{1}_{n}(m,n)\right]
		= \begin{cases}
			m\delta_{n},&\text{ if }n=2,\\
			\frac{1}{2}m\delta_{n},&\text{ if }n\ge 3.\\
		\end{cases}
	$$
\end{lemma}

\begin{proof}
	The first equality of the indices can be obtained by considering the reduction modulo~$n$ which defines  a surjective group homomorphism $:\Hbar^{1}_{N}\left(1,1\right) \to \Hbar^{1}_{n}\left(1,1\right)$.  Since the preimage of $\Hbar^{1}_{n}\left(m,n\right)$ under this reduction is equal to $\Hbar^{1}_{N}\left(m,n\right)$, this gives the desired first equality.

	Next, let us prove the second equality. This equality is verified for $n=2$ through direct computations. Let $n\ge3$. Then, by observing the following identity,
	$$
		\left[\Hbar^{1}_{n}(1,1):\Hbar^{1}_{n}(m,n)\right] = \left[\Hbar^{1}_{n}(1,1):\Hbar^{1}_{n}(1,n)\right] \left[\Hbar^{1}_{n}(1,n):\Hbar^{1}_{n}(m,n)\right],
	$$ it suffices to prove that $\left[\Hbar^{1}_{n}(1,1):\Hbar^{1}_{n}(1,n)\right] = \frac{1}{2}\delta_{n}$ and $\left[\Hbar^{1}_{n}(1,n):\Hbar^{1}_{n}(m,n)\right]=m$. 
	
	Since $n\ge 3$, $-I_{2}$ is in $H^{1}_{n}\left(1,1\right)$ but not in $H^{1}_{n}\left(1,n\right)$ or $H^{1}_{n}\left(m,n\right)$. Therefore, we have  $\left[\Hbar^{1}_{n}\left(1,1\right):\Hbar^{1}_{n}\left(1,n\right)\right] = \frac{1}{2} \left[H^{1}_{n}\left(1,1\right):H^{1}_{n}\left(1,n\right)\right]$ and $\left[\Hbar^{1}_{n}(1,n):\Hbar^{1}_{n}(m,n)\right] = \left[H^{1}_{n}(1,n):H^{1}_{n}(m,n)\right]$. To compute $\left[H^{1}_{n}\left(1,1\right):H^{1}_{n}\left(1,n\right)\right]$, we consider the natural right group action of $\SL{n} = H^{1}_{n}\left(1,1\right)$ on the row vector space $(\bbz/n\bbz)^{2}$, i.e., $$
		(x,y)\begin{pmatrix}a&b\\c&d\end{pmatrix}=(ax+cy,bx+dy).
	$$ The subgroup $H^{1}_{n}\left(1,n\right)$ of $H^{1}_{n}\left(1,1\right)$ is the group of elements that fix $(0,1)$. Therefore, the index $\left[H^{1}_{n}\left(1,1\right):H^{1}_{n}\left(1,n\right)\right]$ is equal to the cardinality of the orbit of $(0,1)$ which is the set of row vectors of order $n$ in $\left(\bbz/n\bbz\right)^{2}$, and this cardinality is $\delta_{n}$,  as shown in the proof of Lemma~\ref{lem:psi_deg}. The equality $\left[H^{1}_{n}\left(1,n\right):H^{1}_{n}\left(m,n\right)\right] = m$ follows from that $$H^{1}_{n}\left(1,n\right) = \left\{\begin{pmatrix}1&b\\0&1\end{pmatrix} \in \SL{n}:b \in \bbz/n\bbz\right\}.$$
\end{proof}

Referring to Lemma~\ref{lem:demo} and Lemma~\ref{lem:spltting_fld}, we embed $\rbar_{\cB}(\Gbar_{N,A,B})$ into $\Hbar_{N}(m,n)$ as follows:

\begin{prop}\label{prop:Gal}
	Let $N\geq 3$ be an integer. For positive divisors $m$ and $n$ of $N $ such that $m \mid n$ and $\left(A,B\right)\in \scrU(K)$, if ${E_{A,B}(K)}[N] \supseteq \bbz/m\bbz\times \bbz/n\bbz $, then the image $\rbar_{\cB}(\Gbar_{N,A,B})$ of the Galois group $\Gbar_{N,A,B}$ of $\Psi_{N,A,B}$ over $K$ under $\rbar_{\cB}$ given in~\eqref{bar-r} is a subgroup of $\Hbar_{N}(m,n)$ for some basis $\cB$ of $E_{A,B}[N]$.
\end{prop}

Proposition~\ref{prop:Gal} relates the Galois group problem for the $N$th primitive division polynomials to the $N$-torsion subgroup problem. Furthermore, Remark~\ref{rmk:psi_N}~(d) states that the Galois groups of the primitive division polynomials are invariant under  so-called ``quadratic twists''. In  Section~\ref{sec:quad_twist}, we formally define quadratic twists of an elliptic curve and show that analyzing the Galois groups of the primitive division polynomials is sufficient to prove the main result, Theorem~\ref{thm:main} via quadratic twists.

\subsection{Quadratic Twists}\label{sec:quad_twist}
The multiplicative group $\Gm=\bba^{1}-\{0\}$ acts on $\bba^{2}$ and $\scrU$ by $D\cdot \left(A,B\right)=(D^{2}A,D^{3}B)$ for $D\in \Gm$ and $\left(A,B\right)$ in $\bba^{2}$ or $\scrU$. We call $D\cdot \left(A,B\right)$ the quadratic twist of $\left(A,B\right)$ by $D$. Equivalently, the quadratic twists of a given curve are defined as follows:
\begin{defn}
	For $\left(A,B\right)\in \bba^{2}$ and $D\in \Gm$, the curve $E_{D^{2}A,D^{3}B} : y^2=x^3+D^2Ax+D^3B$ is called the quadratic twist of $E_{A,B}$ by $D$ and denoted by $E_{A,B}^{D}$. 
\end{defn}

\begin{remark}\label{rmk:twist}$\phantom{s}$
	\begin{enumerate}[\normalfont(a)]
		\item The morphism $\phi:E^{D}_{A,B}\to E_{A,B}$ defined by $(x,y) \mapsto \left(\frac{x}{D},\frac{y}{\sqrt{D^{3}}}\right)$ is an isomorphism defined over $K(\sqrt{D})$. Any curve $E_{A,B}/K$ and its quadratic twist $E_{A,B}^{D}$ is isomorphic over $K$ if and only if $D$ is a square in $K^{\times}$.
		\item Any quadratic twist preserves the Galois group of $\Psi_{N,A,B}$, i.e., $\Gbar_{N,A,B}=\Gbar_{N,D^{2}A,D^{3}B}$, for any $D\in K^\times$ and $\left(A,B\right)\in \scrU(K)$.
		\item The action of every $D^2\in (K^{\times})^{2}$ on $\scrU(K)$ preserves the torsion subgroups, i.e., ${E_{A,B}^{D^{2}}(K)}_{\tors} \cong {E_{A,B}(K)}_{\tors}$ for all $D\in K^{\times}$, since $E_{A,B}^{D^{2}}$ is isomorphic to $E_{A,B}$ over~$K$.
	\end{enumerate}
\end{remark}

We now give a  condition that determines the existence of torsion points of order $m$ and $n$ whose $x$-coordinates and $y$-coordinates are in $K$, by taking quadratic twists.

\

For $\left(A,B\right)\in \scrU(K)$ and two positive integers $m$ and $n$ such that  $m\mid n$, we define the following condition:
\begin{equation} \label{eqn:condP}
	{\textbf{Condition }}\cP(m,n): \Psi_{m,A,B}\text{ splits completely over } K \text{ and } \Psi_{n,A,B}\text{ has a zero in } K.
\end{equation}

	\
	
\begin{prop}\label{prop:Gal_sol}
	Let $N\geq 1$ be an integer, and let $m$ and $n$ be positive divisors of $N$ such that $m\mid n$.  Then for $\left(A,B\right)\in \scrU(K)$, the following three statements are equivalent:
	\begin{enumerate}[\normalfont(a)]
		\item The image $\rbar_{\cB}(\Gbar_{N,A,B})$ of the Galois group $\Gbar_{N,A,B}$ of $\Psi_{N,A,B}$ over $K$ is a subgroup of $\Hbar_{N}(m,n)$ for some basis $\cB$ for $E_{A,B}[N]$.
		\item {\normalfont Condition}~$\cP(m,n)$ is satisfied.
		\item There exists $D\in K^{\times}$ such that $E_{A,B}^{D}(K)[N] \supseteq \bbz/m\bbz\times \bbz/n\bbz$.
	\end{enumerate}
\end{prop}
\begin{proof}
	If $n=1$, all statements hold and therefore, are equivalent. Thus, assume $n\ge 2$.

	To show that (a) implies (b), let $\cB=\{P,Q\}$ be such a basis of $E_{A,B}[N]$ as in (a). From~(a), we have that $$
		r_{\cB}\left(G_{N,A,B}\right) \subseteq \left\langle H_{N}\left(m,n\right), -I_{2}\right\rangle.
	$$ This implies that every element of $G_{N,A,B}$ moves $\tfrac{N}{n}Q$ to $\pm \tfrac{N}{n}Q$. Similarly, any point $R \in E_{A,B}\left[m\right] = \left\{\tfrac{N}{m}a P+\tfrac{N}{m}b Q:a,b\in \bbz/m\bbz\right\}$ is also moved to $\pm R$ by  each element of $G_{N,A,B}$. Consequently, the $x$-coordinates $x\left(\tfrac{N}{n}Q\right)$ and $x(R)$ are elements of $K$. Since $\Psi_{n,A,B}$ has a zero~$x\left(\tfrac{N}{n}Q\right)$, and each zero of $\Psi_{m,A,B}$ is $x\left(R\right)$ for some $R \in E_{A,B}\left[m\right]$ of order $m$, it follows that these zeros are in $K$.
	
	To show that (b) implies (c), let $\alpha \in K$ be a zero of $\Psi_{n,A,B}$. Let $$
		D:= \begin{cases}
			1, & \text{ if } n=2,\\
			\alpha^{3} + A\alpha +B, & \text{ if } n\ge 3\\
		\end{cases}
		\text{ and }
		Q_{n}:= \begin{cases}
			\left(D\alpha,0\right), & \text{ if } n=2,\\
			\left(D\alpha,D^{2}\right), & \text{ if } n\ge 3.
		\end{cases}
	$$ Then $Q_{n}$ is a point of order $n$ in $E_{A,B}^{D}\left(K\right)$ and we can choose a basis $\cB:=\{P, Q\}$ for $E_{A,B}^{D}[N]$ such that $\tfrac{N}{n}Q=Q_{n}$. If $m=1$, then $\bbz/n\bbz \cong \left\langle Q_{n} \right\rangle \subseteq E_{A,B}^{D}\left(K\right)\left[N\right]$, so (c) holds. If $m\ge2$, let $P_{m} := \tfrac{N}{m}P$. Then $x\left(P_{m}\right)$ is a zero of $\Psi_{m,D^{2}A,D^{3}B}$ and since  $\Psi_{m,D^{2}A,D^{3}B}$ splits completely over $K$, it follows that $x\left(P_{m}\right) \in K$. 
	
	Next, we show that $y\left(P_{m}\right)\in K$ as well. If $m=2$, then $y\left(P_{m}\right) = 0\in K$. For $m\ge 3$, since $P_{m}$ and $Q_{m} := \tfrac{N}{m}Q$ form a basis of $E_{A,B}^{D}\left[m\right]$ and the $x$-coordinates of all points of order $m$ in $E_{A,B}^{D}$ lie in $K$, it follows from~\eqref{spltting_fld_arguement} in the proof of Lemma~\ref{lem:spltting_fld} that $y\left(P_{m}\right)y\left(Q_{m}\right) \in K^{\times}$ and $y\left(P_{m}\right) \in K^{\times}$. Hence, we conclude that $y(P_m)\in K^\times$ and so $\bbz/m\bbz\times \bbz/n\bbz \cong \left\langle P_{m}, Q_{n}\right\rangle \subseteq E_{A,B}^{D}(K)[N]$.
	
	Finally, (c) implies (a) directly by Proposition~\ref{prop:Gal} and Remark~\ref{rmk:twist}(b).
\end{proof}

Next, we observe how many $D \in K^{\times}/(K^{\times})^{2}$ satisfy ${E_{A,B}^{D}(K)_{\tors}} \supseteq \bbz/m\bbz \times \bbz/n\bbz$ for a given elliptic curve $E_{A,B}/K$.

\begin{lemma}\label{lem:fixing_twist}
	Let $\left(A,B\right)\in \scrU(K)$ and  let $m$ and $n$ be positive integers such that $m\mid n$. Suppose that $\left(A,B\right)$ satisfies {\normalfont Condition}~$\cP(m,n)$. Then, there exists $D\in K^{\times}/(K^{\times})^{2}$ such that $E^{D}_{A,B}(K)_{\tors} \supseteq \bbz/m\bbz \times \bbz/n\bbz $, and the $y$-coordinate of each point of $E_{A,B}[m]$ is of the form~$a\sqrt{D}$ for some $a\in K$. In particular, we have the following:
	\begin{enumerate}[\normalfont(a)]
		\item If $n\le 2$, then for every $D\in K^{\times}/(K^{\times})^{2}$, we have that $E^{D}_{A,B}(K)_{\tors} \supseteq \bbz/m\bbz \times \bbz/n\bbz $.
		\item If $n\ge 3$, then the number of such  $D$ modulo $(K^{\times})^{2}$ is at most $4$.
		\item If $m\ge 3$, then such $D$ is unique modulo $(K^{\times})^{2}$.
	\end{enumerate}
\end{lemma}
\begin{proof}
	
	The existence of such  $D$  follows from the equivalence of (b) and (c) in Proposition~\ref{prop:Gal_sol}. Considering the isomorphism $\phi: E^{D}_{A,B} \to E_{A,B}$ defined by $(x,y) \mapsto \left(\frac{x}{D},\frac{y}{\sqrt{D^{3}}}\right)$, we conclude that the $y$-coordinates of all $m$-torsion points of $E_{A,B}\left[m\right]$ are of the form $a\sqrt{D}$ for some $a\in K$.

	For (a), if $n=1$, then it holds trivially for every $D \in K^{\times}/(K^{\times})^{2}$,  and if $n=2$, then it is also true, since the $y$-coordinates of points in $E_{A,B}^{D}$ of order $2$ are $0$.
	
	We first prove (c) before proceeding to (b). To prove the uniqueness of such  $D$ modulo~$(K^\times)^2$, suppose there exists another $D'$ satisfying the same conditions. Then, as stated above, the $y$-coordinate of each $m$-torsion points of $E_{A,B}(K)$ must be expressed  in the forms $b\sqrt{D}$ and $b'\sqrt{D'}$ simultaneously for some $b,b'\in K^{\times}$, since $m\ge 3$. This implies that $D'/D\in (K^\times)^2$, proving the uniqueness of $D\in K^{\times}/(K^{\times})^{2}$.

	For (b), assume that for $i=1,2,3,4,5$, there exist five distinct $D_{i}\in K^{\times}/(K^{\times})^{2}$ such that $\bbz/m\bbz \times \bbz/n\bbz \subseteq {E^{D_{i}}_{A,B}(K)}_{\tors}$. Then, for each  $i=1,2,3,4,5$, we define the isomorphism $\phi_{i}: E^{D_{i}}_{A,B} \to E_{A,B}$ by $(x,y) \mapsto \left(\frac{x}{D_{i}},\frac{y}{\sqrt{D_{i}^{3}}}\right)$ over $K(\sqrt{D_{i}})$, and let $R_{i}\in E^{D_{i}}_{A,B}(K)[n]$  be a point of order $n$. Then, we see that $$
		\bbz/m\bbz \times \bbz/n\bbz \subseteq  {E_{A,B}(K(\sqrt{D_{i}}))}_{\tors} \quad \text{ and } \quad 0\ne y(\phi_{i}(R_{i}))=a_{i}\sqrt{D_{i}}\text{ for some }a_{i}\in K^{\times}.
	$$ 
	For each choice of two distinct $i,j\in \{1,2,3,4,5\}$, we have a positive divisor $d_{ij}$ of $n$  such that $\bbz/d_{ij}\bbz \times \bbz/n\bbz \cong \left\langle \phi_{i}(R_{i}),\phi_{j}(R_{j})\right\rangle \subseteq E_{A,B}(K(\sqrt{D_{i}},\sqrt{D_{j}}))$. 
	
	If $d_{ij}=1$ for some distinct  $i$ and $j$, then $\phi_{i}(R_{i})\in \left\langle \phi_{j}(R_{j}) \right\rangle$. Since $n\geq 3$,  the non-zero element $ y(\phi_{i}(R_{i}))$ must be of the form both $a\sqrt{D_{i}}$ and $ b\sqrt{D_{j}}$  for some $a,b\in K^{\times}$. This leads a contradiction, as  $D_i$ and $D_j$ are distinct modulo~$(K^\times)^2$.
	
	If $d_{ij}\ge 3$ for some distinct $i$ and $j$, say $i=1$ and $j=2$, then let $d:=d_{12}$.  Since $ \bbz/d\bbz \times \bbz/d\bbz \subseteq \left\langle \phi_{1}(R_{1}),\phi_{2}(R_{2})\right\rangle$, we have that for $k=3,4,5$, $\frac{n}{d}\phi_{k}(R_{k})\in \left\langle \phi_{1}(R_{1}),\phi_{2}(R_{2})\right\rangle$, and $0\ne y\left(\frac{n}{d}\phi_{k}(R_{k})\right) \in K(\sqrt{D_{1}},\sqrt{D_{2}})$. Therefore, $D_{k}$ is a square in $K(\sqrt{D_{1}},\sqrt{D_{2}})$. In other words, the extension $K\left(\sqrt{D_{k}}\right) / K$ is a subextension $K\left(\sqrt{D_{1}},\sqrt{D_{2}}\right) / K$, whose degree is at most~$2$. The possible subfields of $K\left(\sqrt{D_{1}},\sqrt{D_{2}}\right)$ over $K$ are precisely $K$, $K\left(\sqrt{D_{1}}\right)$, $K\left(\sqrt{D_{2}}\right)$, and $K\left(\sqrt{D_{1}D_{2}}\right)$. Thus, exactly one of the elements $D_{k}$, $\frac{D_{k}}{D_{1}}$, $\frac{D_{k}}{D_{2}}$, and $\frac{D_{k}}{D_{1}D_{2}}$, must be a square in $K$. Since $D_{k}$ is distinct from both $D_{1}$ and $D_{2}$ modulo $(K^\times)^2$, it follows that either $D_{k}$ or~$\frac{D_{k}}{D_{1}D_{2}}$ must be a square in $K$. Then, each of three elements $D_{3}$, $D_{4}$, and $D_{5}$ coincide with one of two elements $1$ and $D_{1}D_{2}$ modulo $(K^\times)^2$. Now, since there are three elements $D_{3}$, $D_{4}$, and $D_{5}$ and only two possible values $1$ and $D_{1}D_{2}$ modulo $(K^\times)^2$, the pigeonhole principle ensures that at least two of $D_{3}$, $D_{4}$, and $D_{5}$  must be the same up to a square in $K$. This contradicts their assumed distinctness.

	If $d_{ij}=2$ for all distinct $i$ and $j$, then $n$ is even and $2\phi_{1}(R_{1})\in \left\langle \phi_{2}(R_{2}) \right\rangle$. Therefore, $y(2\phi_{1}(R_{1}))$ is of form $a\sqrt{D_{1}}$ and $b\sqrt{D_{2}}$ for some $a,b\in K$ simultaneously. Hence, $y(2\phi_{1}(R_{1}))=0$. This is a contradiction, if $n\ge6$, since in this case $y(2\phi_{i}(R_{i}))\ne 0$. If $n=4$, let $\{P,\phi_{1}(R_{1})\}$ be a basis of $E_{A,B}[4]$. Then $\phi_{2}(R_{2})=2P\pm \phi_{1}(R_{1})=\pm \phi_{3}(R_{3})$ since  each of two groups $\left\langle \phi_{1}(R_{1}),\phi_{2}(R_{2})\right\rangle$ and $\left\langle \phi_{1}(R_{1}),\phi_{3}(R_{3})\right\rangle$ is isomorphic to $\bbz/2\bbz \times \bbz/4\bbz\subseteq E_{A,B}[4]$ and both groups contain $\phi_{1}(R_{1})$. But this is a contradiction since $0\ne y(\phi_{2}(R_{2}))=\pm y(\phi_{3}(R_{3}))$ is of the form both $a\sqrt{D_2}$ and $b\sqrt{D_3}$ for some $a,b\in K$, so $y(\phi_{2}(R_{2}))=\pm y(\phi_{3}(R_{3}))=0$ again. So this completes the proof.
\end{proof}

Next, we compare two densities of elliptic curves with prescribed subgroups in terms of $m$ and $n$, and those satisfying Condition~$\cP(m,n)$. The following result shows that the asymptotic behaviors of the former and the latter are practically the same, which highlights the advantage of using Condition~$\cP(m,n)$, as it considers only $x$-coordinates of torsion points without necessarily involving the $y$-coordinates.

\begin{prop}\label{prop:growth}
	For a positive integer $N$ and its positive divisors $m$, $m'$, $n$, and $n'$ such that $m\mid n$, $m'\mid n'$, $m\mid m'$, and $n\mid n'$,  and  $n\ge 3$, we assume that there exists an elliptic curve $E/K$ such that  $E(K)\left[N\right] \supseteq \bbz/m\bbz \times \bbz/n\bbz$.  For a positive real number $X$, let 
	\begin{align*}
		f(X)&=
		\frac
		{\#\{\left(A,B\right)\in \scrU(K):\cH\left(A,B\right)\le X,{E_{A,B}(K)}[N] \supseteq \bbz/m'\bbz \times \bbz/n'\bbz\}}
		{\#\{\left(A,B\right)\in \scrU(K):\cH\left(A,B\right)\le X,{E_{A,B}(K)}[N] \supseteq \bbz/m\bbz \times \bbz/n\bbz\}}\text{ and }\\
		g(X)&=\frac
		{\#\{\left(A,B\right)\in \scrU(K):\cH\left(A,B\right)\le X, \left(A,B\right)\text{ satisfies }\cP(m',n')\}}
		{\#\{\left(A,B\right)\in \scrU(K):\cH\left(A,B\right)\le X, \left(A,B\right)\text{ satisfies }\cP(m,n)\}}.
	\end{align*} Then, for any real number $X>0$, $$ \frac{1}{4}g(X)\le f(X)\le g(X).$$
\end{prop}
\begin{proof}
	For $\left(A,B\right)\in \scrU(K)$, we consider two functions for $X\in\mathbb{R}^{>0},$
	\begin{align*}
		f_{A,B}(X)&=
		\frac
		{\#\{D\in K^{\times}:\cH(D^{2}A,D^{3}B)\le X, ~~~{E_{A,B}^{D}(K)}[N] \supseteq \bbz/m'\bbz \times \bbz/n'\bbz\}}
		{\#\{D\in K^{\times}:\cH(D^{2}A,D^{3}B)\le X,~~~ {E_{A,B}^{D}(K)}[N] \supseteq \bbz/m\bbz \times \bbz/n\bbz\}},\\
		g_{A,B}(X)&=\frac
		{\#\{D\in K^{\times}:\cH(D^{2}A,D^{3}B)\le X,~~~  D\cdot \left(A,B\right)\text{ satisfies }\cP(m',n')\}}
		{\#\{D\in K^{\times}: \cH(D^{2}A,D^{3}B)\le X,~~~ D\cdot \left(A,B\right)\text{ satisfies }\cP(m,n)\}}.
	\end{align*}
	
	Since the set $\left\{ \left(A,B\right)\in \scrU(K) : \cH\left(A,B\right)\le X\right\}$ for each positive number $X$ is finite, we can find a finite subset $\fkR\subsetneq \scrU(K)$ such that any $\left(A,B\right)\in \scrU(K)$ with $\cH\left(A,B\right)\le X$ is written as $\left(A,B\right)=(k^{2}A_{0},k^{3}B_{0})$ where $k\in K^{\times}$, for a unique $(A_{0},B_{0})\in \fkR$. Since the denominator of $f(X)$ is the sum of the denominators of $f_{A,B}(X)$ over all $\left(A,B\right)\in \fkR$. The denominators and numerators of $f(X)$ and $g(X)$ are written similarly. Hence, it is enough to show that $$\frac{1}{4} g_{A,B}(X)\le f_{A,B}(X)\le g_{A,B}(X) \text{ for each }\left(A,B\right)\in \scrU(K).$$	
	
	If $\left(A,B\right)$ does not satisfy Condition~$\cP(m',n')$, then neither does $D\cdot \left(A,B\right)$ for any $D\in K^\times$, and by Remark~\ref{rmk:psi_N}(d) and Proposition~\ref{prop:Gal_sol}, $\bbz/m'\bbz \times \bbz/n'\bbz \not \subseteq {E_{A,B}^{D}(K)}_{\tors}$ for all $D\in K^{\times}$. Hence, it follows that
	\begin{align*}
		&\{D\in K^{\times}: \cH(D^{2}A,D^{3}B)\le X,~~~ D\cdot \left(A,B\right)\text{ satisfies }\cP(m',n')\}
		=\emptyset, \quad \text{ and }\\
	& \{D\in K^{\times}: \cH(D^{2}A,D^{3}B)\le X,~~~ E_{A,B}^{D}(K)[N] \supseteq \bbz/m'\bbz \times \bbz/n'\bbz\}
		=\emptyset,
	\end{align*}
	so $f_{A,B}(X)=g_{A,B}(X)=0$, in which case the statement  holds trivially.
	
	If $\left(A,B\right)$ satisfies Condition~$\cP(m',n')$, then so do $D\cdot \left(A,B\right)$ for all $D\in K^{\times}$ by Proposition~\ref{prop:Gal_sol}. Hence, in this case, \begin{align*}
		\{D\in K^{\times}&: \cH(D^{2}A,D^{3}B)\le X,~~~ D\cdot \left(A,B\right)\text{ satisfies } \cP(m',n')\}\\
		&=
		\{D\in K^{\times}:\cH(D^{2}A,D^{3}B)\le X\}\\
		&=
		\{D\in K^{\times}: \cH(D^{2}A,D^{3}B)\le X,~~~ D\cdot \left(A,B\right)\text{ satisfies }\cP(m,n)\},
	\end{align*} so $g_{A,B}(X)=1$. By Lemma~\ref{lem:fixing_twist} and $n\ge3$, we see  that $$
		\emptyset \neq \{D\in K^{\times}/(K^{\times})^{2}:E_{A,B}^{D}(K)[N] \supseteq \bbz/m'\bbz \times \bbz/n'\bbz\}
	$$ $$
		\subseteq
		\{D\in K^{\times}/(K^{\times})^{2}:E_{A,B}^{D}(K)[N] \supseteq \bbz/m\bbz \times \bbz/n\bbz\},
	$$ and the larger set  (in fact, both sets) of the above has at most four elements. Hence, $$
		\frac{1}{4} \le \frac
		{\#\{D\in K^{\times}:\cH(D^{2}A,D^{3}B)\le X, ~~~E_{A,B}^{D}(K)[N] \supseteq \bbz/m'\bbz \times \bbz/n'\bbz\}}
		{\#\{D\in K^{\times}: \cH(D^{2}A,D^{3}B)\le X, ~~~E_{A,B}^{D}(K)[N] \supseteq \bbz/m\bbz \times \bbz/n\bbz\}}\le 1
	$$ and so $\frac{1}{4}g_{A,B}(X)=\frac{1}{4}\le f_{A,B}(X)\le 1=g_{A,B}(X)$. 
\end{proof}

\begin{remark}
	In the proof of Proposition~\ref{prop:growth}, the finiteness of points in $\scrU(K)$ with bounded height  is important. If there are infinitely many points  in $\scrU(K)$ with bounded height, we need the Axiom of Countable Choice to choose representatives of $\scrU(K)$ up to the group action and get the set $\fkR$. 
\end{remark}

Finally,  we prove that it is enough to consider only Condition~$\cP(m,n)$ instead of a torsion subgroup for our main goal.

\begin{prop}\label{prop:ignoring_twisting}
	Let $N\geq 3$ be an integer. For positive divisors $m$ and $n$ of $N$ such that $m\mid n$, assume that there exists an elliptic curve $E/K$ over $K$ such that $E_{A,B}(K)[N] \supseteq \bbz/m\bbz \times \bbz/n\bbz$. Then  Statement~\ref{Gal_stat} below implies Statement~\ref{tors_stat}. Moreover, if $n\ge 3$, then the following statements~\ref{Gal_stat} and \ref{tors_stat} are equivalent:
	\begin{enumerate}[\normalfont(a)]
		\item\label{Gal_stat} Let $m'$ and $n'$ be positive divisors of $N$ such that $m'\mid n'$, $m\mid m'$, and $n\mid n'$. If almost all elements of the set $\{\left(A,B\right)\in \scrU(K) : \left(A,B\right) \text{ satisfies  {\normalfont Condition} }~\cP(m,n)\}$  satisfy  {\normalfont Condition}~$\cP(m',n')$, then $(m',n')=(m,n)$.
		\item\label{tors_stat} Almost all $\left(A,B\right)\in \scrU(K)$ such that $E_{A,B}(K)[N] \supseteq \bbz/m\bbz \times \bbz/n\bbz$ satisfy $E_{A,B}(K)[N] \cong \bbz/m\bbz \times \bbz/n\bbz$.
	\end{enumerate}
\end{prop}
\begin{proof}
	First, suppose $n\ge 3$. If $(m,n)\ne (m',n')$, then consider $f$ and $g$ given in Proposition~\ref{prop:growth}, which implies that $f(X)$ approaches $0$ as $g(X)$ goes to $0$ and vice versa. Therefore, the following two statements are equivalent: 
	\begin{itemize}
		\item  Condition~$\cP(m',n')$ is not satisfied, for almost all $\left(A,B\right)\in \scrU(K)$ satisfying Condition~$\cP(m,n)$.
		\item  $ E_{A,B}(K)[N] \not\supseteq \bbz/m' \times \bbz/n'\bbz$, for almost all $\left(A,B\right)\in \scrU(K)$ such that $E_{A,B}(K)[N] \supseteq \bbz/m\bbz \times \bbz/n\bbz$.
	\end{itemize}
	Since the number of pairs $(m',n')$ of positive divisors $m'$ and $n'$ of $N$ which satisfy the above divisibility conditions on $m$, $m'$, $n$, and $n'$ is finite,  the proof for $n\ge 3$ is thus complete. 
	
	If $n=1,2$, then for any $\left(A,B\right)\in \scrU(K)$, we know that $E_{A,B}(K)[N] \supseteq \bbz/m\bbz \times \bbz/n\bbz$ if and only if $\left(A,B\right)$ satisfies {\normalfont Condition}~$\cP(m,n)$, since the $y$-coordinate of each non-trivial $2$-torsion point of $E_{A,B}$ is $0$.
\end{proof}

For $\left(A,B\right)\in \scrU(K)$ and each integer $N\ge 3$, we have considered the Galois group~$G_{N,A,B}$ to determine whether $E_{A,B}(K)[N]\supseteq \bbz/m\bbz \times \bbz/n\bbz$ or not. The condition that $r_{\cB}(G_{N,A,B})\subseteq H_{N}(m,n)$ for some basis $\cB$ of $E_{A,B}[N]$ implies $E_{A,B}(K)[N] \supseteq \bbz/m\bbz \times \bbz/n\bbz$ but not $E_{A,B}(K)[N]\cong \bbz/m\bbz \times \bbz/n\bbz$. Therefore, to prove that $E_{A,B}(K)[N]\cong \bbz/m\bbz \times \bbz/n\bbz$, we also need to find the condition to have that $r_{\cB'}(G_{N,A,B})\not\subseteq H_{N}(m',n')$ for any basis $\cB'$ of $E_{A,B}[N]$ and divisors $m'$ and $n'$ of $N$ such that $m'\mid n'$, $m\mid m'$, and $n\mid n'$, and we give such a condition as follows.

\begin{lemma}\label{lem:comparing_gps} For an integer $N\geq 3$, let $m$, $m'$, $n$, and $n'$ be positive divisors of $N$ such that $m\mid n$ and $m'\mid n'$. Then, the inclusion $\Hbar^{1}_{N}(m,n) \subseteq \overline{\gamma} \Hbar_{N}(m',n') \overline{\gamma}^{-1}$ holds for some $\gamma \in H_{N}\left(1,1\right)$ if and only if $m'\mid m$ and $n'\mid n$.
\end{lemma}
\begin{proof}
	The ``if" direction is immediate. To see the ``only if" direction, suppose that $\Hbar^{1}_{N}(m,n) \subseteq \overline{\gamma} \Hbar_{N}(m',n') \overline{\gamma}^{-1}$. This implies that $$
		\left\langle H_{N}^{1}(m,n),-I_{2}\right\rangle \subseteq \left\langle \gamma H_{N}^{1}(m',n')\gamma^{-1},-I_{2}\right\rangle.
	$$ First, we observe that for any positive divisor $k$ of $N$, the subgroup $H^{1}_{N}(k,k)$ is normal in $ H^{1}_{N}\left(1,1\right)$, since it is the kernel of the natural projection $H^{1}_{N}\left(1,1\right)\to H^{1}_{k}\left(1,1\right)$. Hence, we have $$
		\begin{pmatrix}1&m\\0&1\end{pmatrix} \in H_{N}^{1}(m,n) \subseteq \left\langle \gamma H_{N}^{1}(m',n') \gamma^{-1},-I_{2}\right\rangle \subseteq \left\langle \gamma H_{N}^{1}(m',m') \gamma^{-1},-I_{2}\right\rangle = \left\langle H_{N}^{1}(m',m'),-I_{2}\right\rangle
	$$ and $$
		\begin{pmatrix}1&0\\n&1\end{pmatrix} \in H_{N}^{1}(n,n) = \gamma^{-1} H_{N}^{1}(n,n) \gamma \subseteq \left\langle \gamma^{-1} H_{N}^{1}(m,n) \gamma,-I_{2}\right\rangle \subseteq \left\langle H_{N}^{1}(m',n'),-I_{2}\right\rangle.
	$$
	These inclusions occur only if $m'\mid m$ and $n'\mid n$.
	
\end{proof}

\begin{lemma}\label{lem:Gal_sol_a} For positive divisors $m$, $m'$, $n$, and $n'$ of an integer $N\ge 3$ such that $m\mid n$, $m'\mid n'$, $m \mid m'$, and $n \mid n'$. If $\Hbar^{1}_{N}\left(m,n\right) \subseteq \rbar_{\cB}\left(\Gbar_{N,A,B}\right) \subseteq \Hbar_{N}\left(m,n\right)$ for a basis~$\cB$ of $E_{A,B}\left[N\right]$, then $\left(A,B\right)$ satisfies Condition~$\cP\left(m',n'\right)$ only if $\left(m',n'\right) = \left(m,n\right)$.
\end{lemma}
\begin{proof}
	By Lemma~\ref{lem:comparing_gps}, we have $\rbar_{\cB'}\left(\Gbar_{N,A,B}\right) \subseteq \Hbar_{N}\left(m',n'\right)$  for a basis $\cB'$ of $E_{A,B}\left[N\right]$ only if $\left(m',n'\right) = \left(m,n\right)$. Therefore, by the equivalence of (a) and (b) in Proposition~\ref{prop:Gal_sol}, it follows that $\left(A,B\right)$ satisfies Condition~$\cP\left(m',n'\right)$ only if $\left(m',n'\right) = \left(m,n\right)$.
\end{proof}

\section{The proofs of the main results}\label{sec:main}

In Section~\ref{sec:trans}, we have considered only $N$-torsion subgroups $E(K)[N]$ of elliptic curves $E/K$ for a given integer  $N\ge 3$. To consider the full torsion subgroups $E(K)_{\tors}$ over $K$, we apply Merel's theorem (\cite{Merel}) which implies that  $E(K)_{\tors}=E(K)[N_K]$ for some constant integer~$N_{K}\ge 3$ depending only on $K$.

For each $(m,n)$, we find a parameterization $E^{m,n}_{r,u}$ of almost all elliptic curves satisfying Condition~$\cP(m,n)$ and show that the Galois group of the $N_{K}$th primitive polynomial of $E^{m,n}_{r,u}$ is between $\Hbar^{1}_{N_{K}}(m,n)$ and $\Hbar_{N_{K}}(m,n)$. This Galois group condition is the conclusion of Theorem~\ref{thm:min_Gal}.

The parameterizations are given in Section~\ref{sec:param} and the the proofs of Theorem~\ref{thm:main} and Corollary~\ref{cor:easy_cor} are completed in Section~\ref{sec:main_proof}.

\subsection{The trivial torsion subgroup}

We consider the trivial torsion subgroup first. 

Extending the definition of the group homomorphism $\rbar_{\cB}$ in \eqref{bar-r} of Lemma~\ref{lem:spltting_fld} and using the same notation, we define the group homomorphism \[
	\rbar_{\cB}:
	\Gal{\left(K\left(V\right)\left(x\left(\mathcal{A}_{p}\left[N\right]\right)\right)/K\left(V\right)\right)} \to \GL{N}/\left\{\pm I_{2}\right\} 
\] for a given elliptic curve $\mathcal{A}_{p}$ parameterized by a point $p\in V(K)$ of a variety $V/K$ with respect to a basis $\cB$ of $\mathcal{A}_{p}[N]$.

\begin{lemma}\label{lem:irr_Psi}
	For each integer $N\ge 3$, $\Psi_{N}$ is irreducible over the function field $K(\scrU)=K(s,t)$ and $\Hbar^{1}_{N}\left(1,1\right)\subseteq \rbar_{\cB}\left(\Gbar_{N}\right) \subseteq \Hbar_{N}\left(1,1\right)$ for any basis $\cB$ of $E_{s,t}[N]$.
\end{lemma}
\begin{proof}
	Let $\cB$ be a basis of $E_{s,t}[N]$. It is obvious that \[
		\rbar_{\cB}\left(\Gbar_{N}\right) \subseteq \GL{N}\big / \{\pm I_{2}\} = \Hbar_{N}\left(1,1\right).
	\] First, we show that $\Hbar^{1}_{N}\left(1,1\right)=\SL{N}\big / \{\pm I_{2}\}\subseteq \rbar_{\cB}\left(\Gbar_{N}\right)$. For a variable $j$ and an elliptic curve \[
		\scrE_{j}:y^{2}=x^{3} - \frac{27j}{4(j-1728)}x-\frac{27j}{4(j-1728)}
	\] defined over the function field $\bbc\left(j\right)$, we get
	
	\begin{align*}
		r_{\cB}^{-1}\left(\SL{N}\right) & = \Gal\left(\bbc\left(j\right)\left(\scrE_{j}[N]\right)\big /\bbc\left(j\right)\right) \phantom{sssssssss}\text{(by \cite[Corollary~7.5.3]{DS})}\\
		\notag & \subseteq \Gal\left(\bbc\left(s,t\right)\left(E_{s,t}[N]\right)\big/\bbc\left(s,t\right)\right) \phantom{ssss}\left(\begin{aligned}
			&\text{since } \scrE_{j} \text{ is a specialization of } E_{s,t} \\
			&\text{(for example, see \cite{StackExchnge})}
		\end{aligned}\right)\\
		\notag & \subseteq
		\Gal\left(K\left(s,t\right)\left(E_{s,t}[N]\right)\bigm/ K(s,t)\right).
	\end{align*} So we conclude that \[
		\Hbar_N^1\left(1,1\right) \subseteq \rbar_{\cB}\left(\Gbar_{N}\right).
	\]
	
	Since the group $ \Gal\left(K\left(s,t\right)\left(E_{s,t}[N]\right)\big/ K\left(s,t\right)\right)$ permutes transitively the torsion points of $E_{s,t}$ of order~$N$, the group $\Gbar_{N}$ permutes transitively the zeros of $\Psi_{N}$. This implies that $\Psi_{N}$ is irreducible.
\end{proof}

Finally, we prove our main result for the trivial torsion subgroup as a generalization of Theorem~\ref{thm:duke}  to an arbitrary number field $K$.

\begin{theorem}\label{thm:main1}
	For almost all $\left(A,B\right)\in \scrU(K)$, ${E_{A,B}(K)}_{\tors}$ is trivial.
\end{theorem}
\begin{proof}
	By Merel's theorem (\cite{Merel}), there exists an integer $N_{K}\ge 3$ such that ${E_{A,B}(K)}_{\tors}$ is $E_{A,B}(K)[N_{K}]$ for all $\left(A,B\right)\in \scrU(K)$. Theorem~\ref{thm:HIT} and Lemma~\ref{lem:irr_Psi} show that \[
		\Hbar^{1}_{N_{K}}\left(1,1\right) \subseteq \rbar_{\cB}\left(\Gbar_{N_{K},A,B}\right) \subseteq \Hbar_{N_{K}}\left(1,1\right)
	\] for some basis $\cB$ of $E_{A,B}[N_{K}]$, for almost all $\left(A,B\right) \in \scrU(K)$. 
	
	 Next, let $m'$ and $n'$ be positive divisors of $N_{K}$ such that $m'\mid n'$. Then, by Lemma~\ref{lem:Gal_sol_a}, almost all $\left(A,B\right) \in \scrU(K)$ satisfy Condition~$\cP\left(m',n'\right)$ only when $\left(m',n'\right) = \left(1,1\right)$. In other words, Proposition~\ref{prop:ignoring_twisting}\ref{Gal_stat} holds, and therefore,  Proposition~\ref{prop:ignoring_twisting}\ref{tors_stat} holds as well.
\end{proof}

\subsection{Parameterizing elliptic curves whose $j$-invariant $\ne 0,1728$ and which satisfy Condition~$\cP\left(m,n\right)$ for $\left(m,n\right) \in T_{g=0} \setminus \left\{\left(1,1\right)\right\}$}\label{sec:param}

In this subsection, we consider the pairs $(m,n) \in T_{g=0}$ such that $\left(m,n\right)\ne\left(1,1\right)$, and for each such pair $(m,n)$, we parameterize elliptic curves satisfying $\cP\left(m,n\right)$. When doing so, we need to carefully consider the $j$-invariants of these elliptic curves. We denote the $j$-invariant simply as $j$. The elliptic curves with a given $j\ne 0, 1728$ are quadratic twists of each other, meaning there is a unique elliptic curve for each $j\ne 0,1728$ up to quadratic twists. However,  elliptic curves with  $j=0$ or  $j=1728$ are infinitely many, even when considered up to quadratic twists. Therefore, we restrict our parameterization to elliptic curves with  $j \ne 0,1728$, satisfying Condition~$\cP\left(m,n\right)$, and we prove that it is sufficient to prove the main result, Theorem~\ref{thm:main}. To focus on elliptic curves with $j\ne 0,1728$, we introduce the quasi-variety $$\scrU^{*} := \left\{\left(s,t\right) \in \scrU: s, t \ne 0\right\},$$ which serves as the moduli space of the elliptic curves with  $j\ne0,1728$.

For the parameterization, we consider the modular curve $X_{1}(m,n)$ which is introduced in Section~\ref{intro} and the weighted projective space, $$
	\bbp^{d_{0},d_{1},\cdots,d_{k}}:=(\bba^{k+1}-\{0\})\bigm/\sim
$$ for positive integer weights $d_{0},d_{1},\cdots,d_{k}$, where the equivalent relation $\sim$ is given by $$(a_{0},a_{1},\cdots,a_{k})\sim (\lambda^{d_{0}}a_{0},\lambda^{d_{1}}a_{1},\cdots,\lambda^{d_{k}}a_{k}), \text{ for } (a_{0},a_{1},\cdots,a_{k})\in \bba^{k+1}-\{0\},  \text{ and all } \lambda\in \bba^{1}-\{0\}.$$ We denote by $[a_{0}:a_{1}:\cdots:a_{k}]$ the equivalent class of $(a_{0},a_{1},\cdots,a_{k})$. 

\

First, we consider the case when $m=1$.

\begin{prop}\label{prop:U1n}
	For an integer $n \ge 2$, the curve $$
		U_{1,n} := \left\{\left[s:t:x_{2}\right]\in \bbp^{2,3,1}:\Psi_{n,s,t}(x_{2})=0, \left(s,t\right) \in \scrU^{*} \right\}
	$$ is birational to $X_{1}\left(1,n\right)$ over $\bbq$.
\end{prop}
\begin{proof}
	First, the definition of $U_{1,n}$ is well-defined, since the polynomials $\Psi_{n,s,t}(x_{2})$ and $4s^{3}+27t^{2}$ are homogeneous when we assign weights to the variables such that $\mathrm{wt}(x_{2})=1$, $\mathrm{wt}(s)=2$, and $\mathrm{wt}(t)=3$, as noted in Remark~\ref{rmk:psi_N}(d).
	
	We show that the open subset $Y_{1}\left(1,n\right):=\Gamma_{1}\left(1,n\right)\backslash \{z\in \bbc:\Im z>0\}$ of $X_{1}\left(1,n\right)$ is birational to $U_{1,n}$. Following \cite[§1.5]{DS}, we recall the structure of the open set $Y_{1}\left(1,n\right)$ as the moduli space \begin{align*}
		\left\{ (E,Q): 
		\text{ an elliptic curve } E \text{ and } Q\in E \text{ such that } \left|Q\right|=n.
		\right\}\biggm/\sim,
	\end{align*} where the equivalence relation $\sim$ is defined by $(E,Q)\sim(E',Q')$ if and only if there exists an isomorphism $\phi:E\to E'$ as an isogeny such that $\phi(Q)=Q'$ (In the notation of \cite[§1.5]{DS}, we have $\Gamma_{1}\left(1,n\right)=\Gamma_{1}(n)$ and $Y_{1}\left(1,n\right)=Y_{1}(n)$). Clearly, $(E,Q)=(E,-Q)\in Y_{1}\left(1,n\right)$ via the isomorphism $E\to E$ defined by $X\mapsto -X$. We let $Y_{1}^{\circ}\left(1,n\right) \subseteq Y_{1}\left(1,n\right)$ be the open subset consisting of $(E,Q)$ such that $j\left(E\right) \ne 0,1728$. We will show that $Y_{1}^{\circ}$ is isomorphic to $U_{1,n}$.

	We construct an isomorphism $\iota: Y_{1}^{\circ}\left(1,n\right)\to U_{1,n}$ by following \cite[III.10]{Silverman}. Any elliptic curve over a field of characteristic $0$ is isomorphic to $E_{s,t}$ for some $\left(s,t\right) \in \scrU^{*}$. Recall that two elliptic curves $E_{s,t}$ and $E_{s',t'}$ of $j\ne0,1728$ are isomorphic if and only if there is $u\in \Gm$ such that $s'=su^{4}$ and $t'=tu^{6}$. Moreover, each isomorphism $E_{s,t}\to E_{s',t'}$ is of form $(x,y)\mapsto (u^{2}x,\pm u^{3}y)$. Hence, we corresponds $(E_{s,t},Q)\in Y_{1}^{\circ}\left(1,n\right)$ to $[s:t:x(Q)]\in \bbp^{4,6,2}$, uniquely. Composing with the natural morphism $\bbp^{4,6,2}\to \bbp^{2,3,1}$, we have constructed the morphism $\iota:Y_{1}^{\circ}\left(1,n\right)\to \bbp^{2,3,1}$ so far. The image of $\iota$ lying in $U_{1,n}$ because $Q$ is a point of order $n$. We let $\iota:Y_{1}^{\circ}\left(1,n\right)\to U_{1,n}$. We show that the morphism $\iota$ is an isomorphism by constructing the inverse $\kappa: U_{1,n} \to Y_{1}^{\circ}\left(1,n\right)$. For $n=2$, we let $\kappa: \left[s:t:x_{2}\right] \mapsto \left(E_{s, t},\left(x_{2},0\right)\right)$, and for $n\geq 3$, we let $\kappa: \left[s:t:x_{2}\right] \mapsto \left(E_{D^{2}s, D^{3}t},\left(Dx_{2},D^{2}\right)\right)$, where $D := x_{2}^{3}+sx_{2}+t\ne 0$. It is easy to verify that $\iota \circ \kappa$ and $\iota \circ \kappa$ are the identities of $U_{1,n}$ and $Y_{1}^{\circ}\left(1,n\right)$, respectively. Therefore, $Y_{1}^{\circ}\left(1,n\right)$ and $U_{1,n}$ are isomorphic, and  $U_{1,n}$ and $X_{1}\left(1,n\right)$ are birational.
\end{proof}

\begin{prop}\label{prop:U1n_param}
	For $\left(1,n\right) \in T_{g=0}$ such that $n \ge 2$, we assume that there exists an elliptic curve defined over $K$ with $j \ne 0,1728$  whose torsion subgroup contains $\bbz/n\bbz$. Then we have pair-wise relatively prime polynomials $s_{1,n}$, $t_{1,n}$, and $u_{2,1,n}$ in $K\left[r\right]$ such that $\left[s_{1,n}:t_{1,n}:u_{2,1,n}\right]: \bbp^{1} \dashrightarrow U_{1,n}$ defined by $\left[X:Y\right] \mapsto \left[s_{1,n}\left(\tfrac{X}{Y}\right) : t_{1,n}\left(\tfrac{X}{Y}\right) : u_{2,1,n}\left(\tfrac{X}{Y}\right) \right]$ parameterizes $U_{1,n}$ defined in Proposition~\ref{prop:U1n}.
\end{prop}
\begin{proof}
	We write the $n$th primitive division polynomial of $E_{s,t}$ as $\theta_{n}(s,t,x) := \Psi_{n,s,t}(x)\in \bbq [s,t][x]$.
	First, let $n=2,3$. If $\theta_{n}\left(s_{0},t_{0},0\right) = 0$, then the $j$-invariant of $E_{s_{0},t_{0}}$ is $0$ or $1728$. In other words, all the points of $U_{1,n}$ are of form $\left[s:t:1\right]$. Considering the explicit coefficients of $\theta_{n}$ from Definition~\ref{defn:division_polynomial} and Definition~\ref{defn:primitive_division_polynomial}, we parameterize $U_{1,2}$ and $U_{1,3}$ by $\left[r:-r-1:1\right]$ and $\left[r:\frac{r^{2} -6r -3}{12}:1\right]$, respectively. Hence, the proof is completed for $n=2,3$.
	
	For $n\ge 4$, by Remark~\ref{rmk:psi_N}(d) and Lemma~\ref{lem:psi_deg}, $\theta_{n}$ can be written as $$
		\theta_{n}(s,t,x) = \sum_{2i+3j+k=\delta_{n}/2}c_{ijk} s^{i} t^{j} x^{k}\in \bbq [s,t][x],
	$$ where $\delta_{n}$ is given in \eqref{eqn:delta}. We note that $12 \mid \delta_{n}$ by Lemma~\ref{lem:psi_deg}. Now, we show that $$
		c_{0,0,\delta_{n}/2}\ne 0, \quad c_{\delta_{n}/4,0,0}\ne 0, \quad \text{ and } \quad c_{0,\delta_{n}/6,0}\ne 0.
	$$ Obviously, the leading coefficient $c_{0,0,\delta_{n}/2}$ of $\Psi_{n}$ is $1$, not $0$. The remaining nonzero property, $c_{\delta_{n}/4,0,0}\neq 0,$ and $c_{0,\delta_{n}/6,0}\ne 0$ is equivalent to that $s,t\nmid \theta_{n}(s,t,0)\in \bbq[s,t]$. If $s\mid \theta_{n}(s,t,0)$, then $0$ is a zero of $\Psi_{n,0,t}$ and $(0,\sqrt{t})\in E_{0,t}(\overline{K(t)})$ is a torsion point of order $n$. However, a direct computation shows that $x(2(0,\sqrt{t}))=0$, and the point $(0,\sqrt{t})$ is of order $3$, which is a contradiction since $n\geq 4$. If $t\mid \theta_{n}(s,t,0)$, then $0$ is a zero of $\Psi_{n,s,0}$ and $(0,0)\in E_{s,0}(\overline{K(s)})$ is a torsion point of order $n$. Obviously, the order of the point $(0,0)$ is $2$, which is also a contradiction.

	Since $\left(1,n\right) \in T_{g=0}$, $U_{1,n}$ is a genus $0$ curve with a rational point, which corresponds to the elliptic curve with the given rational torsion containing $\bbz/n\bbz$ as assumed via the correspondence in the proof of Proposition~\ref{prop:U1n}. Hence, there exist $f,g,h_{2}\in K\left(r\right)$ such that $\left[f:g:h_{2}\right]: \bbp^{1} \dashrightarrow U_{1,n}$ parameterizes $U_{1,n}$. For an irreducible polynomial $p \in K\left[r\right]$,  let $v_{p}$ denote the discrete valuation on $K\left(r\right)$ at the prime ideal $p K \left[r\right]$. We may assume that \begin{equation}\label{eqn:vp_bd}
		0 \le \min \left\{v_{p}\left(f\right)/2, v_{p}\left(g\right)/3, v_{p}\left(h_{2}\right)\right\}<1 \text{ for all } p
	\end{equation} by replacing $f$, $g$, and $h_{2}$ by $p^{-2k}f$, $p^{-3k}g$, and $p^{-k}h_{2}$, respectively, where $$k :=  \min \left\{\left\lfloor \frac{v_{p}\left(f\right)}{2}\right\rfloor, \left\lfloor \frac{v_{p}\left(g\right)}{3}\right\rfloor, v_{p}\left(h_{2}\right)\right\}.$$ This replacing is legitimate, since $\left[p^{-2k}f:p^{-3k}g:p^{-k}h_{2}\right]=\left[f:g:h_{2}\right]$. We can achieve the inequality in~\eqref{eqn:vp_bd} after a finite number of such replacements, since one of $v_{p}\left(f\right)$, $v_{p}\left(g\right)$, and $v_{p}\left(h_{2}\right)$ is non-zero only for a finite number of primes~$p$.

	On the other hand,
	\begin{align}\label{eqn:two_min_val}
		\notag &\text{ for any irreducible polynomial } p \in K\left[r\right],\\
		&\text{ at least two of the values } v_{p}\left(f\right)/2, v_{p}\left(g\right)/3, \text{ and } v_{p}\left(h_{2}\right)  \text{ attain the minimum.}
	\end{align} If not, for example, $v_{p}\left(h_{2}\right)$ is the unique minimal value among them, i.e., $v_{p}\left(h\right) < v_{p}\left(f\right)/2, v\left(g\right)/3$, then we have $$
		\infty = v_{p}(0)=\left(\sum_{2i+3j+k=\delta_{n}/2}c_{ijk}f^{i}g^{j}h_{2}^{k}\right) = v_{p}\left(c_{0,0,\delta_{n}/2}h_{2}^{\delta_{n}/2}\right) = \frac{\delta_{n}}{2} v_{p}\left(h_{2}\right).
	$$ It happens only if $v_{p}\left(h_{2}\right) = \infty$, but this contradicts that $v_{p}\left(h_{2}\right) < v_{p}\left(f\right)/2, v_{p}\left(g\right)/3$.

	\eqref{eqn:two_min_val} implies that the minimal value among $v_{p}\left(f\right)/2$, $v_{p}\left(g\right)/3$, and $v_{p}\left(h_{2}\right)$ is an integer. By \eqref{eqn:vp_bd}, the minimal value must be $0$. Again by \eqref{eqn:two_min_val}, at least two of $f$, $g$, and $h_{2}$ are not divisible by any $p$. This means that $f$, $g$, and $h_{2}$ are pair-wise relatively prime polynomials.
	
\end{proof}

Proposition~\ref{prop:U1n_param} parameterizes the elliptic curves of $j\ne 0,1728$ satisfying Condition~$\cP(1,n)$ for each $(1,n) \in T_{g=0}$ such that $n \ge 2$. Now, we parameterize all elliptic curves of $j\ne 0,1728$ satisfying Condition~$\cP(m,n)$ for each $(m,n) \in T_{g=0}$ when $m\ge 2$ and $g_{m,n}=0$. 

Throughout this paper, we let $e_{m}$ denote the Weil pairing (see \cite[III.8]{Silverman}) and we fix a primitive $m$th root of unity $\zeta_{m}$.

\begin{lemma}\label{lem:eta} Let $m$ and $n$ be positive integers such that $m\ge 2$ and $m \mid n$ and let 
	\begin{align*}
		W_{m,n}:=\left\{\left[s:t:x_{1}:x_{2}\right]\in \bbp^{2,3,1,1}:
		\Psi_{m,s,t}(x_{1})=\Psi_{n,s,t}(x_{2})=0, \left(s,t\right) \in \scrU^{*}\right\}.
	\end{align*}
	Then, there exists a rational function $$
		\eta_{m,n}:W_{m,n}\to M:=\left\{z\in \bba^{1}: \xi\left(z\right) = 0\right\}
	$$ defined over $\bbq$ such that $$
		\eta_{m,n}\left(s,t,x_{1},x_{2}\right) = e_{m}\left(P_{m},\tfrac{n}{m}Q_{n}\right) + e_{m}\left(P_{m},-\tfrac{n}{m}Q_{n}\right)
	$$ and $$
		\eta_{m,n}\left(D^{2}s,D^{3}t, Dx_{1}, Dx_{2}\right) = \eta_{m,n}\left(s,t,x_{1},x_{2}\right)
	$$ for any $D \in \Gm$, where $P_{m}$ and $Q_{n}$ are points on $E_{s,t}$ with $x$-coordinates $x_{1}$ and $x_{2}$ and $\xi\left(T\right) := \prod \limits_{k\in \left(\bbz/m\bbz\right)/\left\{\pm1\right\}} \left(T - \left(\zeta_{m}^{k}+\zeta_{m}^{-k}\right)\right)$ is the polynomial over $\bbq$.
\end{lemma}
\begin{proof}
	First, the definition of $W_{m,n}$ is well-defined for reasons similar to those in the beginning of the proof of Proposition~\ref{prop:U1n}.
	
	We let $y_{1} = \sqrt{x_{1}^{3}+sx_{1}+t}$ and $y_{2} = \sqrt{x_{2}^{3}+sx_{2}+t}$ be the $y$-coordinates of $P_{m}$ and $Q_{n}$, respectively. Then $$
		\eta_{m,n} := \eta_{m,n}\left(s,t,x_{1},x_{2}\right) = e_{m}\left(P_{m},Q_{m}\right) + e_{m}\left(P_{m},-Q_{m}\right) \in \bbq\left(s,t,x_{1},x_{2}\right)\left(y_{1},y_{2}\right),
	$$ where $Q_{m} := \tfrac{n}{m} Q_{n}$. We show that $\eta_{m,n}\in \bbq(s,t,x_{1},x_{2})$. Every $\sigma$ in the absolute Galois group of $\bbq(s,t,x_{1},x_{2})$ moves $P_{m}$  to $\epsilon_{1} P_{m}$, and $Q_{n}$ to  $\epsilon_{2} Q_{n}$ for some $\epsilon_{1},\epsilon_{2} \in \left\{\pm1\right\}$. Then, we have \begin{align*}
		\eta_{m,n}^{\sigma}
		& = \left(e_{m}\left(P_{m},Q_{m}\right)\right)^{\sigma} + \left(e_{m}\left(P_{m},-Q_{m}\right)\right)^{\sigma} \\
		\notag & = e_{m}\left(P_{m}^{\sigma},Q_{m}^{\sigma}\right) + e_{m}\left(P_{m}^{\sigma},-Q_{m}^{\sigma}\right) \\
		\notag & = \left(e_{m}\left(P_{m},Q_{m}\right)\right)^{\epsilon_{1}\epsilon_{2}} + \left(e_{m}\left(P_{m},Q_{m}\right)\right)^{-\epsilon_{1}\epsilon_{2}}\\
		\notag & = e_{m}\left(P_{m},Q_{m}\right) + e_{m}\left(P_{m},-Q_{m}\right)\\
		\notag & = \eta_{m,n}.
	\end{align*}
We have shown that $\eta_{m,n} \in \bbq(s,t,x_{1},x_{2})$.

	Note that $\eta_{m,n}(s,t,x_{1},x_{2}) = \eta_{m,n}(D^{2}s,D^{3}t,Dx_{1},Dx_{2})$, as the Weil pairing is invariant under any isomorphism, particularly under the isomorphism $E_{s,t}\to E_{s,t}^{D}$ given by $\left(x,y\right)\mapsto \left(Dx,\sqrt{D^{3}}y\right)$.
\end{proof}

\begin{prop}\label{prop:Umn}
	For $\eta_{m,n}$ in Lemma~\ref{lem:eta} and positive integers $m$ and $n$ such that $m\ge 2$ and $m \mid n$, the curve $$
		U_{m,n} := \eta_{m,n}^{-1}\left(\zeta_{m}+\zeta_{m}^{-1}\right)
	$$ is birational to $X_{1}\left(m,n\right)$ over $\bbq\left(\zeta_{m}\right)$.
\end{prop}
\begin{proof}
	We show that the open subset $Y_{1}(m,n)=\Gamma_{1}(m,n)\backslash \{z\in \bbc: \Im z>0\}$ of $X_{1}(m,n)$ is birational to $U_{m,n}$. Following \cite[§1.5]{DS}, we recall the structure of the open set $Y_{1}(m,n)$ as the moduli space \begin{align*}
		\left\{ 
		(E,P_{m},Q_{n}):\begin{aligned}
			&\text{ an elliptic curve } E \text{ with two points } P_{m} \text{ and } Q_{n} \text{ s.t. } \\
			&\left|P_{m}\right|=m, \left|Q_{n}\right|=n, \text{ and } e_{m}\left(P_{m},\tfrac{n}{m}Q_{n}\right) = \zeta_{m}.	
		\end{aligned}\right\}\biggm/\sim,
	\end{align*} where the equivalence relation $\sim$ is defined by $(E,P_{m},Q_{n})\sim(E',P'_{n},Q'_{n})$ if and only if there is an isogeny $\phi:E\to E'$ as an isomorphism such that $\phi(P_{m})=P_{m}'$ and $\phi(Q_{n})=Q_{n}'$. Obviously, we have that $(E,P_{m},Q_{n})=(E,-P_{m},-Q_{n})\in Y_{1}(m,n)$ via the isomorphism $E\to E$ defined by $X\mapsto -X$. We let $Y_{1}^{\circ}\left(m,n\right) \subseteq Y_{1}\left(m,n\right)$ be the open subset consisting of $(E,P,Q)$ such that $j\left(E\right) \ne 0,1728$. 
	
	We show that $Y_{1}^{\circ}\left(m,n\right)$ is isomorphic to $U_{m,n}$. We construct an isomorphism $\iota: Y_{1}^{\circ}(m,n)\to U_{m,n}$. Using a similar argument as in the proof of Proposition~\ref{prop:U1n}, we associate a point $(E_{s,t},P_{m},Q_{n})\in Y_{1}^{\circ}(m,n)$ with $[s:t:x(P_{m}):x(Q_{n})]\in \bbp^{2,3,1,1}$. The image of $\iota$ lies in $W_{m,n}$, since $\left|P_{m}\right| = m$ and $\left|Q_{n}\right| = n$. Moreover, the image of $\iota$ lies in $U_{m,n}$, since letting $Q_{m} := \tfrac{n}{m} Q_{n}$, we have that $$
		\eta_{m,n}\left(s,t,x\left(P_{m}\right),x\left(Q_{n}\right)\right) = e_{m}\left(P_{m},Q_{m}\right) + e_{m}\left(P_{m},-Q_{m}\right) = \zeta_{m}+\zeta_{m}^{-1}.
	$$
	
	Let $\iota:Y_{1}^{\circ}(m,n)\to U_{m,n}$. To prove that $Y_{1}^{\circ}(m,n)$ and $U_{m,n}$ are isomorphic over $\bbq\left(\zeta_{m}\right)$, it suffices to construct the inverse $\kappa: U_{m,n} \to Y_{1}^{\circ}\left(m,n\right)$ of $\iota$ over $\bbq\left(\zeta_{m}\right)$.
	
	For $m=n=2$, if we define $\kappa: U_{2,2} \to Y_{1}^{\circ}\left(2,2\right)$ by $\left[s:t:x_{1}:x_{2}\right] \mapsto \left(E_{s,t},\left(x_{1},0\right),\left(x_{2},0\right)\right)$, then it is clear that $\kappa \circ \iota$ and $\iota \circ \kappa$ are the identities of $Y_{1}^{\circ}\left(2,2\right)$ and $U_{2,2}$, respectively.

	If $m=2<n$, if we define $\kappa: U_{2,n} \to Y_{1}^{\circ}\left(2,n\right)$ by $\left[s:t:x_{1}:x_{2}\right] \mapsto \left(E_{s,t}^{D},\left(Dx_{1},0\right),\left(Dx_{2},D^{2}\right)\right)$ where $D:=x_{2}^{3}+sx_{2}+t$ which is non-zero recalling that $n\ge 3$, then it is clear that $\kappa \circ \iota$ and $\iota \circ \kappa$ are the identities of $Y_{1}^{\circ}\left(2,n\right)$ and $U_{2,n}$, respectively.
	
	If $m \ge 3$, for $\left[s:t:x_{1}:x_{2}\right] \in U_{m,n}$, we let $\mathcal{D} := \left(x_{1}^{3}+sx_{1}+t\right)\left(x_{2}^{3}+sx_{2}+t\right) \in \bbq\left(\zeta_{m}\right)\left(U_{m,n}\right)$ and $D := x_{2}^{3}+sx_{2}+t$ which is non-zero since $n\ge 3$. For two points $P'_{m} := \left(Dx_{1}, D\sqrt{\mathcal{D}}\right) \in E_{s,t}^{D}\left(\bbq\left(\zeta_{m}\right)\left(U_{m,n}\right)\left(\sqrt{\mathcal{D}}\right)\right)$, $Q'_{n} := \left(Dx_{2}, D^{2}\right) \in E_{s,t}^{D}\left(\bbq\left(\zeta_{m}\right)\left(U_{m,n}\right)\right)$, and $z := e_{m}\left(P'_{m},Q'_{m}\right)$, we have $$
		z+z^{-1} = e_{m}\left(P'_{m},Q'_{m}\right) + e_{m}\left(P'_{m},Q'_{m}\right) = \eta_{m,n}(D^{2}s,D^{3}t,Dx_{1},Dx_{2}) = \eta_{m,n}(s,t,x_{1},x_{2})=\zeta_{m}+\zeta_{m}^{-1}
	$$ and $z=\zeta_{m}^{\epsilon'}$ for some $\epsilon' \in \{\pm 1\}$, where $Q'_{m} := \tfrac{n}{m} Q'_{n}$. For  $\sigma$ in the absolute Galois group $G_{\bbq\left(\zeta_{m}\right)\left(U_{m,n}\right)}$ of $\bbq\left(\zeta_{m}\right)\left(U_{m,n}\right)$ and the character $\chi: G_{\bbq\left(\zeta_{m}\right)\left(U_{m,n}\right)} \to \left\{\pm1\right\}$ defined by $\chi\left(\tau\right) = \frac{\sqrt{\mathcal{D}}^{\tau}}{\sqrt{\mathcal{D}}}$, we have $\left(P'\right)^{\sigma} = \chi\left(\sigma\right) P'$ and $\left(Q'\right)^{\sigma} = Q'$. Therefore, we conclude that  $$
		\zeta_{m}^{\epsilon'} = \left(\zeta_{m}^{\epsilon'}\right)^{\sigma} = e_{m}\left(\left(P'_{m}\right)^{\sigma},\left(Q'_{m}\right)^{\sigma}\right) = \zeta_{m}^{\epsilon'\chi\left(\sigma\right)},
	$$ and hence, $\chi$ is identically $1$, and $\sqrt{\mathcal{D}} \in \bbq\left(\zeta_{m}\right)\left(U_{m,n}\right)$ by $m\ge 3$. For $D' := \epsilon' \sqrt{\mathcal{D}} \in \bbq\left(\zeta_{m}\right)\left(U_{m,n}\right)$, we have $\left(E_{s,t}^{D},\left(Dx_{1}, DD'\right),\left(Dx_{2}, D^{2}\right)\right) \in Y_{1}^{\circ}\left(m,n\right)$. We define $\kappa: U_{m,n} \to Y_{1}^{\circ}\left(m,n\right)$ by $\left[s:t:x_{1}:x_{2}\right] \mapsto \left(E_{s,t}^{D},\left(Dx_{1}, DD'\right),\left(Dx_{2}, D^{2}\right)\right)$. It is clear that $\iota \circ \kappa$ is the identity morphism of $U_{m,n}$. To show that $\kappa \circ \iota$ is the identity morphism of $Y_{1}^{\circ}\left(m,n\right)$, for $\left(E_{s,t},\left(x_{1},y_{1}\right),\left(x_{2},y_{2}\right)\right) \in Y_{1}^{\circ}\left(m,n\right)$, we have that $D = x_{2}^{3} + sx_{2} + t = y_{2}^{2}$, and $\left(y_{1}y_{2}\right)^{2} = D\left(x_{1}^{3} + sx_{1} + t\right) = \mathcal{D}$. Thus, $y_{1}y_{2} = \epsilon D'$ for some $\epsilon \in \left\{\pm1\right\}$. Therefore, we have $$
		\left(\kappa \circ \iota \right)\left(E_{s,t},\left(x_{1},y_{1}\right),\left(x_{2},y_{2}\right)\right)
		= \left(E_{s,t}^{y_{2}^{2}},\left(y_{2}^{2}x_{1},\epsilon y_{1}y_{2}^{3}\right),\left(y_{2}^{2}x_{2},y_{2}^{4}\right)\right).
	$$ Considering the Weil paring, $\epsilon =1$, and we have an isomorphism $\phi: E_{s,t} \to E_{s,t}^{y_{2}^{2}}$ defined by $\left(x,y\right) \mapsto \left(y_{2}^{2}x,y_{2}^{3}y\right)$ over $\bbq\left(\zeta_{m}\right)\left(X_{m,n}\right)$. Moreover, $\phi$ maps two points $\left(x_{1},y_{1}\right)$ and $\left(x_{2},y_{2}\right)$ to $\left(y_{2}^{2}x_{1},y_{1}y_{2}^{3}\right)$ and $\left(y_{2}^{2}x_{2},y_{2}^{4}\right)$, respectively. Hence, $\kappa \circ \iota$ is the identity morphism of $Y_{1}^{\circ}\left(m,n\right)$.
\end{proof}

\begin{remark}
	The open subset $Y_{1}(m,n)$ of the modular curve $X_{1}(m,n)$ and the quasi-variety~$U_{m,n}$ do not depend on the choice of a primitive $m$th root $\zeta_{m}$ of unity. Any primitive $m$th root of unity is for some $a\in (\bbz/m\bbz)^{\times}$. Let $b\in (\bbz/m\bbz)^{\times}$ be the inverse of $a$. If $Y_{1}(m,n)'$ is defined by $\zeta_{m}^{a}$ instead of $\zeta_{m}$ in the definition of $Y_{1}(m,n)$, we define the morphism $Y_{1}(m,n)\rightarrow Y_{1}(m,n)'$ by $(E,P,Q)\mapsto (E,aP,Q)$. Obviously, its inverse is the morphism defined by $(E,P,Q)\mapsto (E,bP,Q)$. Similarly, if $U_{m,n}'$ is defined by $\zeta_{m}^{a}$ instead of $\zeta_{m}$ in the definition of $U_{m,n}$, then the morphism $U_{m,n}\rightarrow U_{m,n}'$ defined by $\left[s:t:x_{1}:x_{2}\right]\to \left[s:t:\frac{\phi_{a}}{\psi_{a}^{2}}(x_{1}),x_{2}\right]$ is an isomorphism, since $\frac{\phi_{a}}{\psi_{a}^{2}}(x_{1})$ is the $x$-coordinate of the multiplication by $a$ of a point whose $x$-coordinate is $x_{1}$. (See the explanation below Definition~\ref{defn:division_polynomial} for further details.)
\end{remark}

\begin{prop}\label{prop:Umn_param}
	For  $\left(m,n\right) \in T_{g=0}$ such that $m\ge 2$, we assume that there exists an elliptic curve of $j \ne 0,1728$ defined over $K$ whose torsion subgroup contains $\bbz/m\bbz/ \times \bbz/n\bbz$. Then, there exist polynomials $s_{m,n}$, $t_{m,n}$, $u_{1,m,n}$, and $u_{2,m,n}$ in $K\left[r\right]$ such that $\gcd\left(s_{m,n},t_{m,n}\right) = 1$, and the rational morphism $\left[s_{m,n}:t_{m,n}:u_{1,m,n}:u_{2,m,n}\right]: \bbp^{1} \dashrightarrow U_{m,n}$ defined by $\left[X:Y\right] \mapsto \left[s_{1,n}\left(\tfrac{X}{Y}\right) : t_{1,n}\left(\tfrac{X}{Y}\right) : u_{2,1,n}\left(\tfrac{X}{Y}\right) : u_{2,2,n}\left(\tfrac{X}{Y}\right) \right]$ parameterizes $U_{m,n}$ given in Proposition~\ref{prop:Umn}.
\end{prop}
\begin{proof}
	Since $\left(m,n\right) \in T_{g=0}$, by Proposition~\ref{prop:Umn}, there exist $f,g,h_{1},h_{2}\in K\left(r\right)$ such that $\left[f:g:h_{1}:h_{2}\right]: \bbp^{1} \dashrightarrow U_{m,n}$ parameterizes $U_{m,n}$.
	
	First, we consider $n = 2,3$. We let $$
        s_{2,2} = -\left(r^{2}+r+1\right), \quad t_{2,2} = r\left(r+1\right), \quad u_{1,2,2} = r, \quad u_{2,2,2} = 1,
    $$ and $$
        s_{3,3} = 3r\left(3r^{2}+3r+1\right), \quad t_{3,3} = \frac{\left(3r^{2}-1\right)\left(9r^{4}+18r^{3}+18r^{2}+6r+1\right)}{4},u_{1,3,3} = r, \quad u_{2,3,3} = 1.
    $$ Then, direct computation shows that $u_{1,m,n}$ and $u_{2,m,n}$ are zeros of $\Psi_{n,s_{m,n},t_{m,n}}$. Since $n=2,3$, we have $\left[s_{m,n}:t_{m,n}:u_{1,m,n}:u_{2,m,n}\right] \in U_{m,n}$. Moreover, this $\left[s_{m,n}:t_{m,n}:u_{1,m,n}:u_{2,m,n}\right]$ is a parmaeterization of $U_{m,n}$. Indeed, if it were not a parameterization, then it would coincide with $\left[f:g:h_{1}:h_{2}\right] \circ \varphi$ for some $\varphi \in K\left(r\right)$ of $\deg\left(\varphi: \bbp^{1} \dashrightarrow \bbp^{1} \right) \ge 2$. Since $u_{2,m,n}=1$, we obtain $$
		\left[\frac{f}{h_{2}^{2}}:\frac{g}{h_{2}^{2}}:\frac{h_{1}}{h_{2}}:1\right] \circ \varphi = \left[f:g:h_{1}:h_{2}\right] \circ \varphi = \left[s_{m,n}:t_{m,n}:u_{1,m,n}:1\right],
	$$ which leads to the contradiction that $\deg \varphi \mid \deg u_{1,m,n} = 1$.

	Now we let $n\ge 4$. As in the proof of Proposition~\ref{prop:U1n_param}, we have that
	\begin{align}\label{eqn:two_min_val2}
		\notag &\text{ for any irreducible polynomial } p \in K\left[r\right],\\
		&\text{ at least two of the values } v_{p}\left(f\right)/2, v_{p}\left(g\right)/3, \text{ and } v_{p}\left(h_{2}\right)  \text{ attain the minimum,}
	\end{align} and we may assume that\begin{align}\label{eqn:vp_bd'}
		0 \le \min \left\{v_{p}\left(f\right)/2, v_{p}\left(g\right)/3, v_{p}\left(h_{1}\right), v_{p}\left(h_{2}\right)\right\}<1 \text{ for all } p.
	\end{align} To prove $f$ and $g$ are relatively prime, it is enough to show that either $v_{p}\left(f\right)$ or $v_{p}\left(g\right)$ is $0$ for any $p$. Suppose that both of them are positive for some irreducible polynomial $p$. By \eqref{eqn:two_min_val2}, the minimal value among $v_{p}\left(f\right)/2$, $v_{p}\left(g\right)/3$, and $v_{p}\left(h_{2}\right)$ is a positive integer. \eqref{eqn:vp_bd'} implies $v_{p}\left(h_{1}\right) = 0$, which contradicts the congruence $$
		0 = \Psi_{m,f,g}\left(h_{1}\right) \equiv h_{1}^{\deg_{x} \Psi_{m,f,g}} \pmod{pK\left[r\right]}.
	$$
\end{proof}

We have parameterized the elliptic curves with $j\ne 0,1728$ that satisfy Condition~$\cP\left(m,n\right)$ for $\left(m,n\right) \in T_{g=0} \setminus \left\{\left(1,1\right)\right\}$. Next, in Theorem~\ref{thm:min_Gal}, we verify that the Galois group associated with $N$-torsion subgroup of the elliptic curve, given in \eqref{eqn:Emn} below and parameterizing Condition~$\cP\left(m,n\right)$, lies between $\Hbar_{N}^{1}\left(m,n\right)$ and $\Hbar_{N}\left(m,n\right)$.
To prove Theorem~\ref{thm:min_Gal}, we will use the following well-known fact.
\begin{lemma}[{\cite[§13.2. Exercises~\#18]{Dummit-Foote}}]\label{lem:fld_theory}
	Let $F$ be a field and $u$ be an indeterminate. For two non-zero relatively prime polynomials $f,g\in F[X]$, let $r_{0} \in \overline{F(u)}$ be a solution of the equation $u=\frac{f(X)}{g(X)}$ for $X$. Then, we have $[F(r_{0}):F(u)]=\max\{\deg f,\deg g\}$.
\end{lemma}

\begin{theorem}\label{thm:min_Gal}
	For  $\left(m,n\right) \in T_{g=0} \setminus \left\{\left(1,1\right)\right\}$, assume that there exists an elliptic curve defined over $K$ with $j \ne 0,1728$  whose torsion subgroup contains $\bbz/m\bbz \times \bbz/n\bbz$. Then, there are two relatively prime polynomials $f$ and $g$ in $K\left[r\right]$ satisfying the following properties; \begin{enumerate}[\normalfont(a)]
		\item For any positive $N$ which is divisible by $n$ and for some (ordered) basis $\cB'$ of $E^{m,n}_{r,u}[N]$, $$
			\Hbar^{1}_{N}(m,n) \subseteq \rbar_{\cB'}\left(\Gbar'_{N}\right) \subseteq \Hbar_{N}(m,n),
		$$ where $\Gbar'_{N} := \Gal\left(K\left(\mathscr{W}\right)\left(E^{m,n}_{r,u}\left[N\right]\right)/K\left(\mathscr{W}\right)\right)$ is the Galois group of $N$-torsion of the elliptic curve \begin{equation}\label{eqn:Emn}
			E^{m,n}_{r,u}:y^{2}=x^{3}+u^{2}f(r)x+ u^{3}g(r)
		\end{equation} over the function field $K(\mathscr{W})$, where $\mathscr{W}:=\{(r,u)\in \bba^{2}: u\ne 0, (f\left(r\right),g\left(r\right)) \in \scrU^{*}\}$.
		\item $\max\left\{3\deg f, 2\deg g\right\} = \begin{cases}
			m\delta_{n},&\text{ if }n=2,\\
			\frac{1}{2}m\delta_{n},&\text{ if }n\ge 3, \\
		\end{cases}$ where $\delta_{n}$ is given in \eqref{eqn:delta}.
		\item All but finitely many elliptic curves with $j\ne 0,1728$ satisfying {\normalfont Condition}~$\cP\left(m,n\right)$ (up to quadratic twists) are specializations of $E^{m,n}_{r,u}$.
	\end{enumerate}
	
\end{theorem}
\begin{proof}
	The proof for the case $m = 1$ is similar to, yet quite simpler than the proof for the cases when $m \ge 2$. Hence, we assume that $m\ge2$. We take $f=s_{m,n}$ and $g=t_{m,n}$ be $s_{m,n}$ as given in Proposition~\ref{prop:Umn_param}. Note that  $\gcd\left(f,g\right) = 1$ by Proposition~\ref{prop:Umn_param}.
	
	To prove (a), we show that there exist a non-zero $D \in K\left(r\right)$ and two linearly independent points $P_{m}$ of order $m$ and~$Q_{n}$ of order $n$ in $\scrE_{j}^{D}\left(K\left(r\right)\right)$, where $j := 1728\frac{4f^{3}}{4f^{3}+27g^{2}} \in K\left(r\right)$ and $$
		\scrE_{j} = E^{m,n}_{r,f/g}: y^{2}=x^{3}-\frac{27j}{4(j-1728)}x-\frac{27j}{4(j-1728)}.
	$$ 
	Since $E^{m,n}_{r,u}$ is a quadratic twist of $\scrE_{j}$ by $u g/f \in K(r,u)$ and the equation of $\scrE_{j}$ has no coefficients involving $u$, we have \begin{align*}
		\Gbar'_{N} & = \Gal\left(K(r,u)(x(E^{m,n}_{r,u}[N]))\big/ K(r,u)\right) \\
		\notag & = \Gal\left(K(r,u)(x(\scrE_{j}[N]))\big/ K(r,u)\right)\\
		\notag & = \Gal\left(K(r)(x(\scrE_{j}[N]))\big/ K(r)\right).
	\end{align*} Now, we will verify that $\Gal\left(K(r)(x(\scrE_{j}[N]))\big/ K(r)\right)$ lies between $\Hbar^{1}_{N}(m,n)$ and $\Hbar_{N}(m,n)$. By Proposition~\ref{prop:Umn_param},  for the $k$th primitive division polynomial $\Phi_{k,r,u}\left(x\right)$ of $E^{m,n}_{r,u}$, we have two polynomials $h_{1}$ and $h_{2}$ in $K\left[r\right]$ such that \begin{itemize}
		\item $h_{1}$ is a zero of $\Phi_{m,r,1}\left(x\right)$,
		\item $h_{2}$ is a zero of $\Phi_{n,r,1}\left(x\right)$, and
		\item $\eta_{m,n}\left(f,g,h_{1},h_{2}\right) = \zeta_{m} + \zeta_{m}^{-1}$.
	\end{itemize}
We let $$
		D := \begin{cases}
			1, & \text{ if } n=2,\\
			\left(\frac{fh_{2}}{g}\right)^{3}+\frac{f^{3}}{g^{2}}\frac{fh_{2}}{g}+\frac{f^{3}}{g^{2}}, & \text{ if } n\ge3
		\end{cases}
		\text{ and }
		Q_{n} := \begin{cases}
			\left(D\frac{fh_{2}}{g}, 0\right), & \text{ if } n=2,\\
			\left(D\frac{fh_{2}}{g}, D^{2}\right), & \text{ if } n\ge3.
		\end{cases}
	$$ Then, $Q_{n}$ is a point of order $n$ in $\scrE_{j}^{D}\left(K\left(r\right)\right)$, since $Dfh_{2}/g$ is a zero of the $n$th primitive division polynomial $\Phi_{n,r,Df/g}$ of $\scrE_{j}^{D} = \left(E^{m,n}_{r,1}\right)^{Df/g}$. Similarly, we have a point $P_{m} \in \scrE_{j}^{D}\left(\overline{K\left(r\right)}\right)$ of order $m$ with $x$-coordinate $Dfh_{1}/g$.
	
	We show that $P_{m} \in \scrE_{j}^{D}\left(K\left(r\right)\right)$. If $m=2$, then $P_{m} \in \scrE_{j}^{D}\left(K\left(r\right)\right)$ since $y\left(P_{m}\right)=0$. If $m\ge 3$, we let $Q_{m} := \tfrac{n}{m} Q_{n}$ and $z := e_{m}\left(P_{m},Q_{m}\right)$, where $e_{m}$ denotes the Weil pairing. Since $\scrE_{j}^{D} = E_{f^{3}/g^{2},f^{3}/g^{2}}^{D}$, we have $$
		z + z^{-1} = \eta_{m,n}\left(D^{2}f^{3}/g^{2},D^{3}f^{3}/g^{2},Dfh_{1}/g,Dfh_{2}/g\right) = \eta_{m,n}\left(f,g,h_{1},h_{2}\right) = \zeta_{m} + \zeta_{m}^{-1},
	$$ which implies that $z$ is $\zeta_{m}$ or $\zeta_{m}^{-1}$. Since $x\left(P_{m}\right) \in K\left(r\right)$, 
	we see that for any $\sigma$ in the absolute Galois group of $K\left(r\right)$, we have that $P_{m}^{\sigma} = \epsilon P_{m}$ for some $\epsilon \in \left\{\pm1\right\}$. Since we assume  the existence of an elliptic curve over $K$ whose torsion subgroup over $K$ contains $\bbz/m\bbz \times \bbz/n\bbz$, it follows that $\zeta_{m} \in K$.
	Hence, we have $$
		z = z^{\sigma} = \left(e_{m}\left(P_{m},Q_{m}\right)\right)^{\sigma} = e_{m}\left(P_{m}^{\sigma},Q_{m}^{\sigma}\right) = e_{m}\left(\epsilon P_{m},Q_{m}\right) = z^{\epsilon},
	$$ so we conclude $\epsilon = 1$ recalling that $m\ge 3$, and we have $P_{m} \in \scrE_{j}^{D}\left(K\left(r\right)\right)$. Since the Weil pairing $z$ of $P_{m}$ and $Q_{m}$ is a primitive $m$th root of unity, $P_{m}$ and $Q_{n}$ are linearly independent. By Lemma~\ref{lem:abgp}, there is a basis $\cB = \left\{P,Q\right\}$ of $\scrE_{j}^{D}\left[N\right]$ such that $\tfrac{N}{m}P = P_{m}$ and $\tfrac{N}{n}Q = Q_{n}$. We let $P',Q' \in \scrE_{j}$ be the images of $P,Q \in \scrE_{j}^{D}$, respectively, under the isomorphism from $\scrE_{j}^{D} \rightarrow \scrE_{j}$ defined by $\left(x,y\right) \mapsto \left(\frac{x}{D},\frac{y}{\sqrt{D^{3}}}\right) $. Then, $\cB' = \left\{P',Q'\right\}$  is a basis of $\scrE_{j}$, and we have \begin{align*}
		\rbar_{\cB'}\left(\Gbar'_{N}\right) & =\rbar_{\cB'}\left(\Gal\left(K\left(r\right)\left(x\left(\scrE_{j}[N]\right)\right) {\LARGE/} K\left(r\right)\right)\right)\\
		\notag & = \rbar_{\cB}\left(\Gal\left(K\left(r\right)\left(x\left(\scrE_{j}^{D}[N]\right)\right) {\LARGE/} K\left(r\right)\right) \right) \\
		\notag & \subseteq \Hbar_{N}\left(m,n\right).
	\end{align*}
Next,	Now we show that $\Hbar^{1}_{N}(m,n) \subseteq \rbar_{\cB'}\left(\Gbar'_{N}\right)$. Since $\gcd\left(f,g\right)=1$, $j\in \bbc(r)$ is transcendental over $\bbc$, and therefore,  \cite[Corollary~7.5.3]{DS} implies that \begin{equation*}
		\rbar_{\cB'}\left(\Gal(\bbc\left(j\right)\left(x\left(\scrE_{j}[N]\right)\right)\big/\bbc\left(j\right))\right)
		=\Hbar^{1}_{N}\left(1,1\right),
	\end{equation*}
	and 
	\begin{equation}\label{univ_Gal}
		\rbar_{\cB'}\left(\Gal\left(\bbc\left(j\right)\left(x\left(\scrE_{j}[N]\right)\right)\big/\bbc\left(j\right)\left(x\left(\tfrac{N}{m}P'\right),x\left(\tfrac{N}{n}Q'\right)\right)\right)\right)
		=\Hbar^{1}_{N}(m,n)
	\end{equation}
	by recalling the definition of $\rbar_{\cB'}$ over $\bbc\left(j\right)$ and that the image of $\rbar_{\cB'}$ is contained in $\SL{N}$.
	
	On the other hand, we have \begin{equation}\label{eqn:jx1x2}
		j = 1728 \frac{4f^{3}/g^{2}}{4f^{3}/g^{2}+27}, \quad x\left(\tfrac{N}{m}P'\right) = \frac{fh_{1}}{g}, \quad \text{ and } \quad x\left(\tfrac{N}{n}Q'\right) = \frac{fh_{2}}{g}.
	\end{equation} Therefore, we have the following diagram of subfields of  $\bbc\left(r\right)\left(x\left(\scrE_{j}[N]\right)\right)$ over $\bbc(j)$:
\begin{equation}\label{eqn:diagram}
\begin{tikzcd}[cramped, sep=small]
																												 & {\bbc\left(r\right)\left(x\left(\scrE_{j}[N]\right)\right)}																		&											  \\
																												 &																																				& \bbc\left(r\right) \arrow[lu, no head] \\
{\bbc\left(j\right)\left(x\left(\scrE_{j}[N]\right)\right)} \arrow[rd, no head] \arrow[ruu, no head] &																																				&											  \\
																												 & {\bbc\left(j\right)\left(x\left(\scrE_{j}[N]\right)\right) \cap \bbc\left(r\right)} \arrow[d, no head] \arrow[ruu, no head] &											  \\
																												 & {\bbc\left(j\right)\left(x\left(\tfrac{N}{m}P'\right),x\left(\tfrac{N}{n}Q'\right)\right)} \arrow[d, no head]	 &											  \\
																												 & \bbc\left(j\right)																													  &											 
\end{tikzcd}\end{equation}
Now, we show that \begin{equation}\label{eqn:fld_equal}
		\bbc(r)
		= \bbc\left(j\right)\left(x\left(\tfrac{N}{m}P'\right),x\left(\tfrac{N}{n}Q'\right)\right).
	\end{equation}
	The inclusion `$\supseteq$' has already been explained. We prove the other inclusion `$\subseteq$'.
The birational morphism $\left[f:g:h_{1}:h_{2}\right]: \bbp^{1} \dashrightarrow U_{m,n}$ over $K$ has a rational function as its inverse, $\tau: U_{m,n} \dashrightarrow \bbp^{1}$. By \eqref{eqn:jx1x2}, $$
		r = \tau\left(\left[f:g:h_{1}:h_{2}\right]\right) = \tau\left(\left[\frac{f^{3}}{g^{2}}:\frac{f^{3}}{g^{2}}:\frac{fh_{1}}{g}:\frac{fh_{2}}{g}\right]\right) \in \bbc\left(j\right)\left(x\left(\tfrac{N}{m}P'\right),x\left(\tfrac{N}{n}Q'\right)\right),
	$$ which implies that \eqref{eqn:fld_equal} holds.
	Also, considering the diagram~\eqref{eqn:diagram}, \eqref{eqn:fld_equal} implies that the  intermediate field $\bbc\left(j\right)\left(x\left(\scrE_{j}[N]\right)\right)~\cap~\bbc\left(r\right)$ must be equal to $ \bbc\left(j\right)\left(x\left(\tfrac{N}{m}P'\right),x\left(\tfrac{N}{n}Q'\right)\right)$. Therefore, we have the following:
	\begin{align*}
		\Gbar'_{N} = 
		\Gal\left(K\left(r\right)\left(x\left(\scrE_{j}[N]\right)\right) {\LARGE/} K\left(r\right)\right)
		 & \supseteq
		\Gal\left(\bbc\left(r\right)\left(x\left(\scrE_{j}[N]\right)\right) {\LARGE/} \bbc\left(r\right)\right) \\
		\notag & \cong
		\Gal\left(\bbc\left(j\right)\left(x\left(\scrE_{j}[N]\right)\right) {\LARGE/} \bbc\left(j\right)\left(x\left(\scrE_{j}[N]\right)\right) \cap \bbc\left(r\right)\right)\\
		\notag & =
		\Gal\left(\bbc\left(j\right)\left(x\left(\scrE_{j}[N]\right)\right) {\LARGE/} \bbc\left(j\right)\left(x\left(\tfrac{N}{m}P'\right),x\left(\tfrac{N}{n}Q'\right)\right)\right) \\
		\notag & = \rbar_{\cB'}^{-1}\left(\Hbar^{1}_{N}(m,n)\right),
	\end{align*}
	which implies that  $\Hbar^{1}_{N}(m,n) \subseteq \rbar_{\cB'}\left(\Gbar'_{N}\right)$.

	Moreover, by Lemma~\ref{lem:fld_theory}, \eqref{eqn:fld_equal}, \eqref{univ_Gal}, and Lemma~\ref{lem:gpidx} together with $\gcd\left(f,g\right) = 1$, the part~(b) can be verified as \begin{align*}
		\max\{3\deg f, 2\deg g\} = \left[\bbc(r):\bbc(j)\right] & = \left[\bbc\left(j\right)\left(x\left(\tfrac{N}{m}P'\right),x\left(\tfrac{N}{n}Q'\right)\right):\bbc(j)\right] \\
		\notag &= \left[\Hbar^{1}_{N}\left(1,1\right):\Hbar^{1}_{N}(m,n)\right]= \begin{cases}
			m\delta_{n},&\text{ if }n=2,\\
			\frac{1}{2}m\delta_{n},&\text{ if }n\ge 3.\\
		\end{cases}
	\end{align*}

	The part (c) has already been established in Proposition~\ref{prop:Umn_param}.
\end{proof}

Now, we  prove that considering only elliptic curves of $j \ne 0,1728$ is enough to prove the main result, Theorem~\ref{thm:main} as follows.

\begin{lemma}\label{lem:j01728}
	Let $n'\ge 2$ be an integer. Then, there exist finite subsets $\mathscr{A}_{n'},\mathscr{B}_{n'} \subseteq \overline{K}^{\times}$ such that for every positive divisor $m'$ of  $ n'$, \begin{align*}
		\left\{a\in K^{\times}: (a,0) \text{ satisfies }\cP(m',n')\right\} & \subseteq  \bigcup_{a \in \mathscr{A}_{n'} \cap K} a(K^{\times})^{2}, \text{ if } \left(m',n'\right) \ne \left(1,2\right) \text{ and }\\
		\left\{b\in K^{\times}: (0,b) \text{ satisfies }\cP(m',n')\right\} & \subseteq  \bigcup_{b \in \mathscr{B}_{n'} \cap K} b(K^{\times})^{3}, \text{ if } \left(m',n'\right) \ne \left(1,3\right).
	\end{align*}
\end{lemma}
\begin{proof}
	We prove the existence of $\mathscr{A}_{n'}$. Then, the existence of $\mathscr{B}_{n'}$ can be proved similarly.

	If $n'\ge 4$, we observe that the $n'$th primitive division polynomial $\Psi_{n',s,0}(x)\in \bbq[s][x]$  of $E_{s,0}$ with degree $\delta_{n'}/2$ is $F(x^{2},s)$ for some homogeneous polynomial $F(X,s)\in \bbq[X,s]$ which is not divisible by $X$ or $s$. This observation follows by the following facts:\begin{itemize}
		\item Neither $x$ nor $s$ divides $\Psi_{n',s,0}(x)$ (as in the proof of Proposition~\ref{prop:U1n}).
		\item The polynomial $\Psi_{n',s,0}(x)\in \bbq[s][x]=\bbq[x,s]$ is homogeneous when assigning weights to the variables $x$ and $s$ such that $\mathrm{wt}(x)=1$ and $\mathrm{wt}(s)=2$.
	\end{itemize}
	Since $X,s\nmid F(X,s)$, we have that $F(X,s)=\prod_{i=1}^{\delta_{n'}/4}(X-\alpha_{i}s)$ for some $\alpha_{i} \in \overline{K}^{\times}$. In other words, $\Psi_{n',s,0}(x)=\prod_{i=1}^{\delta_{n'}/4}(x^{2}-\alpha_{i}s)$, and any $(a,0)$ satisfying $\cP(m',n')$ must have  $a \in a_{*} \left(K^{\times}\right)^{2}$, where $a_{*}$ belongs to the finite set $\mathscr{A}_{n'} := \left\{\pm \frac{1}{\sqrt{\alpha_{i}}}: i=1,\cdots,\delta_{n'}/4\right\}$.

	If $n'=3$, by a direct computation, the 3rd primitive division polynomial $\Psi_{3,s,0}(x)\in \bbq[s][x]$ of degree $4$ of $E_{s,0}$ is $F(x^{2},s)$, where $F(X,s)\in \bbq[X,s]$ is a homogeneous polynomial that is not divisible by $X$ or $s$. As in the case $n'\geq 4$, we similarly obtain a finite set $\mathscr{A}_{3}$ in this case.

	If $n'=2$, then we have $(m',n')=(2,2)$. If for some  $a \in K^{\times}$, $E_{a,0}$  satisfies Condition~$\cP\left(2,2\right)$, meaning that the 2nd primitive division polynomial $x^{3} + ax$ of $E_{a,0}$ has a zero other than~$0$, then it follows that $a = -r^{2}$ for some $r \in K^{\times}$. Consequently, the set $\mathscr{A}_{2} := \left\{ -1 \right\}$ satisfies the statement.
	
\end{proof}

\subsection{Proof of the main theorem}\label{sec:main_proof}

Finally, we complete the proofs of our main results, Theorem~\ref{thm:main} and Corollary~\ref{cor:easy_cor}.

\begin{proof}[Proof of Theorem~\ref{thm:main}]
	If $(m,n)=\left(1,1\right)$, then Theorem~\ref{thm:main1} completes the proof.
	
	Suppose $(m,n)\in T_{g=0}$ given in~\eqref{eqn:genus0} with $(m,n)\neq \left(1,1\right)$. By Merel's theorem (\cite{Merel}), there exists a positive integer $N_{K}\ge 4$ such that ${E_{A,B}(K)}_{\tors}=E_{A,B}(K)[N_{K}]$ for all $\left(A,B\right)\in \scrU(K)$. Then, $n$ divides $N_{K}$ since there exists an elliptic curve $E/K$ such that ${E(K)}_{\tors} \supseteq \bbz/m\bbz \times \bbz/n\bbz$ by assumption. The existence of Merel's constant admits only finite number of pairs $\left(m',n'\right)$ of positive integers satisfying that $m \mid m'$, $n \mid n'$, $m'\mid n'$, $\left(m',n'\right) \ne \left(m,n\right)$, and there exists an elliptic curve $E/K$ over $K$ such that ${E\left(K\right)}_{\tors} \supseteq \bbz/m'\bbz \times \bbz/n'\bbz$. 
	
	Recalling Proposition~\ref{prop:ignoring_twisting}, it suffices to show that Proposition~\ref{prop:ignoring_twisting}\ref{Gal_stat}, which asserts that if $(m',n')\neq (m,n)$, then the function $G$ of positive real numbers $X$ defined by 
	\begin{align}\label{eqn:goal}
		\notag & G\left(X\right)  := \frac
		{\#\left\{\left(A,B\right) \in \scrU\left(K\right): \cH\left(A,B\right) \le X,~~ \left(A,B\right) \text{ satisfies }  \cP\left(m',n'\right) \right\}}
		{\#\left\{\left(A,B\right) \in \scrU\left(K\right): \cH\left(A,B\right) \le X,~~ \left(A,B\right) \text{ satisfies }  \cP\left(m,n\right) \right\}}\\
		& \text{converges to } 0 \text{ as } X \to \infty,
	\end{align} where the height function $\cH$ is introduced in the list of notations before Definition~\ref{defn:almostall}. Remark~\ref{rmk:psi_N}(d) implies that Condition~$\cP\left(m,n\right)$  is invariant under the quadratic twist, i.e., for any $D \in K^{\times}$ and $\left(A,B\right) \in \scrU\left(K\right)$, \begin{equation}\label{eqn:condP_not_vary_under_twist}
		D \cdot \left(A,B\right) \text{ satisfies } \cP\left(m,n\right) \text{ if and only if } \left(A,B\right) \text{ satisfies }\cP\left(m,n\right).
	\end{equation}
	Then, we have an invariant set $\cF$ under the quadratic twist as follows;$$
		\cF := \left\{\left(A,B\right) \in \scrU\left(K\right): \left(A,B\right) \text{ satisfies Condition } \cP\left(m,n\right) \right\}.
	$$ For any  subset $\cF'\subseteq \cF $ which is invariant under the quadratic twist, we call a  subset  $\fkR'$ of $\cF'$ a complete set of representatives of $\cF'$, if each $\left(A,B\right)\in \cF'$ can be  written in the form $\left(A,B\right)=(k^{2}A_{0},k^{3}B_{0})$, for some $k\in K^{\times}$ and a unique $(A_{0},B_{0})\in \fkR'$. 
	
	Now, we analyze  the function $G$ given in \eqref{eqn:goal} in terms of elements of a complete set of representatives of~$ \cF$.
	
	Let $\fkR$ be a complete set of representatives  of $ \cF$. For  $\left(A_{0},B_{0}\right) \in \fkR$ and a real number $X$  such that $X > \inf\left\{\cH\left(D^{2}A_{0},D^{3}B_{0}\right): D \in K^{\times}\right\}$, we define a function
	\begin{align*}
		G_{A_{0},B_{0}}(X)& =\frac
		{\#\left\{D\in K^{\times}:\cH(D^{2}A_{0},D^{3}B_{0})\le X,~~~  D\cdot \left(A_{0},B_{0}\right)\text{ satisfies } \cP(m',n') \right\}}
		{\#\left\{D\in K^{\times}: \cH(D^{2}A_{0},D^{3}B_{0})\le X,~~~ D\cdot \left(A_{0},B_{0}\right)\text{ satisfies } \cP(m,n) \right\}}.
	\end{align*} 
  The denominator and numerator of $G$ (as given in \eqref{eqn:goal}) are the sums of the denominators and numerators of $G_{A_{0},B_{0}}$, respectively, where $\left(A_{0},B_{0}\right)$ runs over $\fkR$. This function $G_{A_{0},B_{0}}$ provides insight into $G$ as follows: If $$
		X > \inf \left\{\cH\left(A,B\right) : \left(A,B\right) \in \cF\right\},
	$$then the set $$
		\fkR_{X} := \left\{\left(A_{0},B_{0}\right) \in \fkR: \cH\left(D^{2}A_{0},D^{3}B_{0}\right) \le X \text { for some } D \in K^{\times}\right\}
	$$ is not empty. Furthermore, $\fkR_{X}$ is finite, since the set $\left\{\left(A,B\right) \in \scrU\left(K\right): \cH\left(A,B\right) \le X \right\}$ is finite. For $\left(A_{0},B_{0}\right) \in \fkR_{X}$, we have $$
		\left\{D\in K^{\times}: \cH(D^{2}A_{0},D^{3}B_{0})\le X,~~~ D\cdot \left(A_{0},B_{0}\right)\text{ satisfies } \cP(m,n) \right\} \ne \emptyset,
	$$ and thus, \begin{align*}
		G(X)& =\frac
		{\sum_{\left(A_{0},B_{0}\right) \in \fkR_{X}}\#\left\{D\in K^{\times}:\cH(D^{2}A_{0},D^{3}B_{0})\le X,~~~  D\cdot \left(A_{0},B_{0}\right)\text{ satisfies } \cP(m',n') \right\}}
		{\sum_{\left(A_{0},B_{0}\right) \in \fkR_{X}}\#\left\{D\in K^{\times}: \cH(D^{2}A_{0},D^{3}B_{0})\le X,~~~ D\cdot \left(A_{0},B_{0}\right)\text{ satisfies } \cP(m,n) \right\}}.
	\end{align*} By \eqref{eqn:condP_not_vary_under_twist}, for $\left(A_{0},B_{0}\right) \in \fkR$, $G_{A_{0},B_{0}}$ is a constant function and its constant value is $$
		G_{A_{0},B_{0}}\left(X\right) = \begin{cases}
			0, & \text{ if } \left(A_{0},B_{0}\right) \text{ does not satisfy }\cP\left(m',n'\right),\\
			1, & \text{ if } \left(A_{0},B_{0}\right) \text{ satisfies }\cP\left(m',n'\right).
		\end{cases}
	$$ Therefore, it follows that \begin{equation}\label{eqn:G_refined}
		G\left(X\right) \le \frac{\#\left\{\left(A_{0},B_{0}\right) \in \fkR_{X}: \left(A_{0},B_{0}\right) \text{ satisfies } \cP\left(m',n'\right) \right\}}{\# \fkR_{X}}.
	\end{equation} To prove \eqref{eqn:goal}, let $$
		\cF^{*} := \left\{\left(A,B\right) \in \scrU^{*}\left(K\right): \left(A,B\right) = \left(f\left(r\right)u^{2},g\left(r\right)u^{3}\right) \text{ for } r,u \in K\right\},
	$$ where the polynomials $f,g\in K[r]$ are as given in Theorem~\ref{thm:min_Gal}. Then, we note that $\cF^{*}$ is a subset of $\cF$, and is invariant under the quadratic twist. We now carefully choose a complete set of representatives $\fkR^{*}$ of $\cF^{*}$ as follows: Let $J := \left\{1728\frac{4f^{3}}{4f^{3}+27g^{2}}\left(r\right): r \in K \setminus S \right\}$, where  $S := \left\{ r\in K: 1728\frac{4f^{3}}{4f^{3}+27g^{2}}\left(r\right) \in \left\{0,1728,\infty \right\}\right\}$ is a finite set. By Theorem~\ref{thm:min_Gal}(b), the rational morphism $1728\frac{4f^{3}}{4f^{3}+27g^{2}}: \bbp^{1} \dashrightarrow \bbp^{1}$ has degree $d := \begin{cases}m\delta_{n},&\text{ if }n=2\\\frac{1}{2}m\delta_{n},&\text{ if }n\ge3\\\end{cases}$. In other words, for each $j \in J$, the preimage $\left(1728\frac{4f^{3}}{4f^{3}+27g^{2}}\right)^{-1}\left(j\right)$ is non-empty and contains at most $d$ elements. 
	For each $j \in J$, we choose an element $\rho_{j} \in \left(1728\frac{4f^{3}}{4f^{3}+27g^{2}}\right)^{-1}\left(j\right)$ that minimizes the height, i.e., $\cH\left(f\left(\rho_{j}\right), g\left(\rho_{j}\right)\right) \le \cH\left(f\left(r\right), g\left(r\right)\right)$ for all $r \in \left(1728\frac{4f^{3}}{4f^{3}+27g^{2}}\right)^{-1}\left(j\right)$, and then we show that the set $\left\{\left(f\left(\rho_{j}\right), g\left(\rho_{j}\right)\right) \in \cF^{*}: j \in J\right\}$ can be chosen as $\fkR^{*}$, i.e., $$
		\fkR^{*} = \left\{\left(f\left(\rho_{j}\right), g\left(\rho_{j}\right)\right) \in \cF^{*}: j \in J\right\}.
	$$ Note that any element in $\cF^{*}$ can be written as $\left(f\left(r\right)u^{2},g\left(r\right)u^{3}\right)$ for some $r \in K \setminus S$ and $u \in K^{\times}$. For any $r \in K \setminus S$ and $u \in K^{\times}$, the pair $\left(f\left(r\right)u^{2},g\left(r\right)u^{3}\right)$ is a quadratic twist of $\left(f\left(\rho_{j}\right),g\left(\rho_{j}\right)\right) \in \fkR^{*}$, where $j := 1728 \frac{4f^{3}}{4f^{3}+27g^{2}}\left(r\right)$. Conversely, if $\left(f\left(r\right)u^{2},g\left(r\right)u^{3}\right)$ is a quadratic twist of $\left(f\left(\rho_{j}\right),g\left(\rho_{j}\right)\right)$ for some $j \in J$, then $j = 1728\frac{4f^{3}}{4f^{3}+27g^{2}}\left(\rho_{j}\right) = 1728\frac{4f^{3}}{4f^{3}+27g^{2}}\left(r\right)$. Summarizing, $\left\{\left(f\left(\rho_{j}\right), g\left(\rho_{j}\right)\right) \in \cF^{*}: j \in J\right\}$ forms a complete set  of representatives $\fkR^{*}$ of $\cF^{*}$.

	For a given positive real number $X$, let $$
		\fkR^{*}_{X} := \fkR_{X} \cap \fkR^{*},
	$$ and we define a  function $$
		F_{X}: \left\{r \in K \setminus S: \cH\left(f\left(r\right),g\left(r\right)\right) \le X\right\} \to \fkR^{*}_{X},
	$$ by $F_{X}\left(r\right) := \left(f\left(\rho_{j}\right), g\left(\rho_{j}\right)\right)$, where $j = 1728 \frac{4f^{3}}{4f^{3}+27g^{2}}\left(r\right)$. The condition defining the domain of $F_{X}$ ensures that $\cH\left(f\left(\rho_{j}\right), g\left(\rho_{j}\right)\right) \le \cH\left(f\left(r\right), g\left(r\right)\right) \le X$ and $\left(f\left(\rho_{j}\right), g\left(\rho_{j}\right)\right) \in \fkR^{*}_{X}$.

	Next, we show that \begin{equation}\label{eqn:FX}
		1 \le \# F_{X}^{-1}\left(\left\{\left(A,B\right)\right\}\right) \le d \text{ for each } \left(A,B\right) \in \fkR^{*}_{X}.
	\end{equation}
	Each $\left(A,B\right) \in \fkR^{*}_{X}$ satisfies $\left(A,B\right) = \left(f\left(\rho_{j}\right), g\left(\rho_{j}\right)\right)$ for $j \in J$, which implies $\rho_{j} \in F_{X}^{-1}\left(\left\{\left(A,B\right)\right\}\right)$. This guarantees that $1 \le \# F_{X}^{-1}\left(\left\{\left(A,B\right)\right\}\right)$. Now, suppose $r_{0} \in F_{X}^{-1}\left(\left\{\left(A,B\right)\right\}\right)$. Then, we have $$1728\frac{4f^{3}}{4f^{3}+27g^{2}}\left(r_{0}\right) = j_{A,B} := 1728\frac{4A^{3}}{4A^{3}+27B^{2}} \ne 0, 1728.$$ Thus, $r_{0}$ is a preimage of $j_{A,B}$ under the rational morphism $1728\frac{4f^{3}}{4f^{3}+27g^{2}}: \bbp^{1} \dashrightarrow \bbp^{1}$ of degree~$d$. Therefore, it follows that $\# F_{X}^{-1}\left(\left\{\left(A,B\right)\right\}\right) \le d$.
	
	Now, we let $\fkR_{0}$ be a complete set of representatives  of $\cF \setminus \cF^{*}$ (which is invariant under the quadratic twist) so that $\fkR = \fkR_{0} \cup \fkR^{*}$ is a complete set of representatives of $\cF$. We then partition $\fkR_{0}$ into three subsets as follows:
	\begin{align*}
		\fkR^{\operatorname{A}} & := \left\{\left(A_{0},B_{0}\right) \in \fkR_{0}: B_{0} = 0\right\},\\
		\fkR^{\operatorname{B}} & := \left\{\left(A_{0},B_{0}\right) \in \fkR_{0}: A_{0} = 0\right\}, \text{and} \\
		\fkR^{\operatorname{except}} & := \fkR_{0} \setminus \left(\fkR^{\operatorname{A}} \cup \fkR^{\operatorname{B}} \right).
	\end{align*}
	
	Since $\left(m,n\right) \ne \left(1,1\right), \left(m',n'\right)$, it follows that $\left(m',n'\right) \ne \left(1,2\right),\left(1,3\right)$. Then, by Lemma~\ref{lem:j01728}, two sets $$
		\mathscr{A} := \left\{\left(A_{0},0\right)\in \fkR^{\operatorname{A}}: \left(A_{0},0\right) \text{ satisfies }\cP\left(m',n'\right)\right\},
	$$  and $$
		\mathscr{B} := \left\{\left(0,B_{0}\right) \in \fkR^{\operatorname{B}}: \left(0,B_{0}\right) \text{ satisfies }\cP\left(m',n'\right)\right\}
	$$ are finite. Furthermore, by Theorem~\ref{thm:min_Gal}(c) and \eqref{eqn:condP_not_vary_under_twist}, the set $\fkR^{\operatorname{excep}}$ is also finite. Thus, the union $\mathscr{A} \cup \mathscr{B} \cup \fkR^{\operatorname{excep}}$ is finite. Let $X$ be a real number such that $$
		X > \inf \left\{\cH\left(A,B\right): \left(A,B\right) \in \cF^{*}\right\}.
	$$ Then, $\fkR^{*}_{X}$ is a non-empty finite subset of $ \fkR^{*}$, and referring to \eqref{eqn:G_refined}, we have:
	\begin{align*}
		G\left(X\right) & \le \frac{\#\left\{\left(A_{0},B_{0}\right) \in \fkR_{X}: \left(A_{0},B_{0}\right) \text{ satisfies } \cP\left(m',n'\right) \right\}}{\# \fkR^{*}_{X}}\\
		\notag & \le \frac{\#\left(\left(\mathscr{A} \cup \mathscr{B} \cup \fkR^{\operatorname{excep}}\right)\cap \fkR_{X}\right) + \#\left\{\left(A_{0},B_{0}\right) \in \fkR^{*}_{X}: \left(A_{0},B_{0}\right) \text{ satisfies } \cP\left(m',n'\right) \right\}}{\# \fkR^{*}_{X}}\\
		\notag & \le \frac{\#\left(\mathscr{A} \cup \mathscr{B} \cup \fkR^{\operatorname{excep}} \right) + \#\left\{\left(A_{0},B_{0}\right) \in \fkR^{*}_{X}: \left(A_{0},B_{0}\right) \text{ satisfies } \cP\left(m',n'\right) \right\}}{\# \fkR^{*}_{X}}.
	\end{align*}
	Since $\fkR^{*}$ is infinite and the union $\mathscr{A} \cup \mathscr{B} \cup \fkR^{\operatorname{excep}}$ is finite, we conclude that \begin{equation}\label{eqn:G_bd_1st}
		\lim_{X \to \infty} G\left(X\right) \le \lim_{X \to \infty} \frac{\#\left\{\left(A_{0},B_{0}\right) \in \fkR^{*}_{X}: \left(A_{0},B_{0}\right) \text{ satisfies } \cP\left(m',n'\right) \right\}}{\# \fkR^{*}_{X}}.
	\end{equation}
	Now, we prove \begin{equation}\label{eqn:R*X}
		\fkR^{*}_{X} = \left\{\left(A,B\right) \in \fkR^{*}: \cH\left(A,B\right) \le X\right\}, \text{ for } X \gg 0.
	\end{equation} By \cite[VIII.Theorem~5.6]{Silverman} and Theorem~\ref{thm:min_Gal}(b), the quantities $\frac{\cH\left(f\left(r\right),g\left(r\right)\right)}{H\left(r\right)^{d}}$ and $\frac{\cH\left(u^{2}f\left(r\right),u^{3}g\left(r\right)\right)}{H\left(r,u\right)^{d+6}}$ are bounded both below and above by positive constants, since the morphisms $\phi_{1}: \bba^{1} \to \bba^{2}$ defined by $\phi_{1}\left(r\right) = \left(f^{3}\left(r\right),g^{2}\left(r\right)\right)$, and $\phi_{2}: \bba^{2} \rightarrow \bba^{2}$ defined by $\phi_{2}\left(r, u\right) = \left(\left(u^{2}f\left(r\right)\right)^{3},\left(u^{3}g\left(r\right)\right)^{2}\right)$ have degree $d+6$. Therefore, for some positive constant~$C$, we have $$
		\frac{\cH\left(f\left(r\right),g\left(r\right)\right)}{\cH\left(f\left(r\right)u^{2},g\left(r\right)u^{3}\right)} \le C \frac{H\left(r\right)^{d}}{H\left(r,u\right)^{d+6}} \le C \frac{H\left(r\right)^{d}}{H\left(r\right)^{d+6}} = \frac{C}{H\left(r\right)^{6}},
	$$ which implies that there exists a constant $H_{\operatorname{bd}}$ such that \begin{equation}\label{eqn:large_Hr}
		\text{ if } H\left(r\right) > H_{\operatorname{bd}}, \text{then } \cH\left(f\left(r\right),g\left(r\right)\right) \le \cH\left(f\left(r\right)u^{2},g\left(r\right)u^{3}\right) \text{ for any } u \in K^{\times}.
	\end{equation}
	For a real number $X$ such that \begin{equation}\label{eqn:large_X}
		X > \max \left\{ \cH\left(f\left(r\right),g\left(r\right)\right): r\in K, H\left(r\right) \le H_{\operatorname{bd}} \right\},
	\end{equation} we have \begin{align}
		\label{eqn:fgr_large_r}&\left\{\left(f\left(r\right),g\left(r\right)\right) \in\scrU^{*}\left(K\right): r \in K, \cH\left(f\left(r\right)u^{2},g\left(r\right)u^{3}\right) \le X \text{ for some } u \in K^{\times} \right\} \\
		\notag& = \left\{\left(f\left(r\right),g\left(r\right)\right) \in \scrU^{*}\left(K\right): r \in K, H\left(r\right)\le H_{\operatorname{bd}}, \cH\left(f\left(r\right)u^{2},g\left(r\right)u^{3}\right) \le X \text{ for some } u \in K^{\times}\right\} \\
		\notag& \phantom{==}\cup \left\{\left(f\left(r\right),g\left(r\right)\right) \in \scrU^{*}\left(K\right): r \in K, H\left(r\right)> H_{\operatorname{bd}}, \cH\left(f\left(r\right)u^{2},g\left(r\right)u^{3}\right) \le X \text{ for some } u \in K^{\times} \right\}\\
		\notag& = \left\{\left(f\left(r\right),g\left(r\right)\right) \in \scrU^{*}\left(K\right): r \in K, H\left(r\right)\le H_{\operatorname{bd}}, \cH\left(f\left(r\right),g\left(r\right)\right) \le X \right\} \phantom{sssss} \text{(by \eqref{eqn:large_X})}\\
		\notag& \phantom{==}\cup \left\{\left(f\left(r\right),g\left(r\right)\right) \in \scrU^{*}\left(K\right): r \in K, H\left(r\right)> H_{\operatorname{bd}}, \cH\left(f\left(r\right),g\left(r\right)\right) \le X \right\} \phantom{ss} \text{(by taking } u=1\text{ and \eqref{eqn:large_Hr})}\\
		\notag& = \left\{\left(f\left(r\right),g\left(r\right)\right) \in \scrU^{*}\left(K\right): r \in K, \cH\left(f\left(r\right),g\left(r\right)\right) \le X \right\}.  
	\end{align}
	Consequently, \eqref{eqn:R*X} holds, since $\fkR^{*}_{X}$ is a subset of the set in \eqref{eqn:fgr_large_r}.

	Considering $F_{X}$, and using \eqref{eqn:FX} and\eqref{eqn:R*X}, we see that \eqref{eqn:G_bd_1st} implies that \begin{align}
		\lim_{X \to \infty} G\left(X\right) & \le \lim_{X \to \infty} \frac{ d \cdot \#\left\{r \in K \setminus S: \left(f\left(r\right),g\left(r\right)\right) \text{ satisfies } \cP\left(m',n'\right) \text{ and } \cH\left(f\left(r\right),g\left(r\right)\right) \le X\right\}}{\#\left\{r \in K \setminus S:  \cH\left(f\left(r\right),g\left(r\right)\right) \le X\right\}}.
	\end{align} Since there exist two positive constants $C_{1}$ and $C_{2}$ such that $C_{1} \le \frac{\cH\left(f\left(r\right),g\left(r\right)\right)}{H\left(r\right)^{d}} \le C_{2}$,  we obtain the inequality,
	\begin{align}\label{eqn:G_bd}
		\frac{1}{d} \lim_{X \to \infty} G\left(X\right) &
		\le \lim_{X \to \infty} \frac{\#\left\{r \in K \setminus S: \left(f\left(r\right),g\left(r\right)\right) \text{ satisfies } \cP\left(m',n'\right) \text{ and } H\left(r\right)^{d} \le C_{1}^{-1} X\right\}}{\#\left\{r \in K \setminus S:  H\left(r\right)^{d} \le C_{2}^{-1} X\right\}}.
	\end{align} 
	
	Then, it suffices to show that the right hand side of \eqref{eqn:G_bd} is $0$, to complete the proof of our goal \eqref{eqn:goal}. Since $$
		\Gal\left(K\left(r\right)\left(x\left(E^{m,n}_{r,1}\left[N_{K}\right]\right)\right) / K\left(r\right) \right) = \Gal\left(K\left(r,u\right)\left(x\left(E^{m,n}_{r,u}\left[N_{K}\right]\right)\right) / K\left(r,u\right) \right),
	$$ and $E_{f\left(r\right),g\left(r\right)} = E^{m,n}_{r,1}$, Theorem~\ref{thm:min_Gal}(a) and Theorem~\ref{thm:HIT} establish that almost all $r_{0} \in K \setminus S$ satisfy $$
		\Hbar^{1}_{N_{K}}(m,n) \subseteq \rbar_{\cB}\left(\Gal\left(K\left(x\left(E_{f\left(r_{0}\right),g\left(r_{0}\right)}\left[N_{K}\right]\right)\right) / K \right)\right) \subseteq \Hbar_{N_{K}}(m,n),
	$$ for some basis $\cB$ of $E_{f\left(r_{0}\right),g\left(r_{0}\right)}\left[N_{K}\right]$. Then, Lemma~\ref{lem:Gal_sol_a} implies that for almost all $r_{0} \in K \setminus S$, $E_{f\left(r_{0}\right),g\left(r_{0}\right)}$ does not satisfy $\cP\left(m',n'\right)$,  since $(m,n)\neq (m',n')$. In other words, $$
		\lim_{X \to \infty} \frac{\#\left\{r \in K \setminus S: \left(f\left(r\right),g\left(r\right)\right) \text{ satisfies } \cP\left(m',n'\right) \text{ and } H\left(r\right)^{d} \le C_{1}^{-1} X\right\}}{\#\left\{r \in K \setminus S:  H\left(r\right)^{d} \le C_{1}^{-1} X\right\}} = 0.
	$$ Furthermore, since \cite[\S2.5 Theorem(Schanuel)]{Serre97} shows that the value $$
		\frac{\#\left\{r \in K \setminus S:  H\left(r\right)^{d} \le C_{1}^{-1} X\right\}}{\#\left\{r \in K \setminus S:  H\left(r\right)^{d} \le C_{2}^{-1} X\right\}} \text{ is bounded,}
	$$ it follows that the right hand side of \eqref{eqn:G_bd} is $0$, completing the proof.
	
\end{proof}

\begin{proof}[Proof of Corollary~\ref{cor:easy_cor}]
	First, suppose that $\zeta_{m}\in K$. For each $(m,n) \in T_{g=0} \setminus \left\{(5,5)\right\}$, \cite{LMFDB} provides an elliptic curve $E/\bbq(\zeta_{m})$ with $j\ne0,1728$ such that $$
		{E(\bbq(\zeta_{m}))}_{\tors} \supseteq \bbz/ m \bbz \times \bbz / n \bbz.
	$$ Examples of such elliptic curves are provided in Table~\ref{table:exmaple} below. 
	\begin{table}[!h]
		\centering
		\begin{tabular}{|c||c|c|c|c|c|c|c|c|}
			\hline
			$\left(m,n\right)$ & $\left(1,7\right)$ & $\left(1,9\right)$ & $\left(1,10\right)$ & $\left(1,12\right)$ & $\left(2,6\right)$ & $\left(2,8\right)$ & $\left(3,6\right)$ & $\left(4,4\right)$ \\\hline
			elliptic curve & 26.b2 & 54.b2 & 66.c3 & 90.c7 & 30.a6 & 210.e6 & 196.2-a3 & 200.2-a3\\\hline
		\end{tabular}
		\captionsetup{justification=centering}
		\caption{\centering \label{table:exmaple}Examples of $E/\bbq(\zeta_{m})$ such that ${E(\bbq(\zeta_{m}))}_{\tors} \supseteq \bbz/ m \bbz \times \bbz / n \bbz$}
	\end{table}
	
	We note that although Table~\ref{table:exmaple} provides examples for only eight cases, it effectively covers all~$\left(m,n\right) \in T_{g=0} \setminus \left\{\left(5,5\right)\right\}$. For instance, while $\left(1,3\right) \in T_{g=0}$ is not explicitly listed in Table~\ref{table:exmaple}, the torsion part of the elliptic curve labeled 54.b2 (corresponding to the case $(1,9)$) contains~$\bbz/3\bbz$, thereby covering this case.
	
	For $(m,n)=(5,5)$, by \cite[Theorem~3.3]{J4}, there exists an elliptic curve $E/\bbq(\zeta_{5})$ such that ${E(\bbq(\zeta_{5}))}_{\tors} \supseteq \bbz/ 5 \bbz \times \bbz / 5 \bbz$, since the modular curve $X_{1}(5,5)$ is the moduli space of the elliptic curves whose Mordell-Weil groups of rational points contain $\bbz/5\bbz \times \bbz/5\bbz$.

	The converse follows directly from \cite[III.Corollary~8.1.1]{Silverman}.
\end{proof}


\begin{thebibliography}{9}

\bibitem[BN]{Bruin_Najman} Bruin, Peter and Najman, Filip, \textit{Counting elliptic curves with prescribed level structures over number fields},
J. Lond. Math. Soc. (2) \textbf{105} (2022), no. 4, 2415--2435.

\bibitem[DF]{Dummit-Foote} Dummit, David S. and Foote, Richard M., \textit{Abstract algebra, 3rd ed.}(John Wiley \& Sons, Inc., Hoboken, NJ 2004).


\bibitem[DS05]{DS} Diamond, Fred and Shurman, Jerry,
\textit{A first course in modular forms}(Springer-Verlag, New York, 2005).

\bibitem[DS04]{DS56} Derickx, Maarten and Sutherland, Andrew V.,
\textit{Torsion subgroups of elliptic curves over quintic and sextic number fields},
Proc. Amer. Math. Soc. \textbf{145} (2017), no. 10, 4233--4245.

\bibitem[Du97]{Duke} Duke, William,
\textit{Elliptic curves with no exceptional primes},
C. R. Acad. Sci. Paris S{\' e}r. I Math. \textbf{325} (1997), no. 8, 813--818.

\bibitem[F83]{Faltings}
Faltings, Gerd, \textit{Endlichkeitss{\" a}tze f{\" u}r abelsche Variet{\" a}ten {\" u}ber Zahlk{\" o}rpern. (German) [Finiteness theorems for abelian varieties over number fields]}, Invent. Math. 73 (1983), no. 3, 349--366.

\bibitem[Gu21]{Gu21}
Gu{\v z}vi{\' c}, Tomislav,
\textit{Torsion growth of rational elliptic curves in sextic number fields},
J. Number Theory 220 (2021), 330--345.

\bibitem[HS17]{Harron_Snowden} Harron, Robert and Snowden, Andrew,
\textit{Counting elliptic curves with prescribed torsion},
J. Reine Angew. Math. \textbf{729} (2017), 151--170.

\bibitem[JKS04]{J3} Jeon, Daeyeol, Kim, Chang Heon and Schweizer, Andreas,
\textit{On the torsion of elliptic curves over cubic number fields},
Acta Arith. \textbf{113} (2004), no. 3, 291--301. 

\bibitem[JKP06]{J4} Jeon, Daeyeol, Kim, Chang Heon and Park, Euisung,
\textit{On the torsion of elliptic curves over quartic number fields},
J. London Math. Soc. (2) \textbf{74} (2006), no. 1, 1--12. 

\bibitem[J08]{J5} Jeon, Daeyeol, \textit{Tetragonal modular curves $X_{1}(M,N)$}, Commun. Korean Math. Soc. \textbf{23} (2008), no. 3, 343--348

\bibitem[K91]{Kamienny} Kamienny, Sheldon,
\textit{Torsion points on elliptic curves and $q$-coefficients of modular forms},
Invent. Math. \textbf{109} (1992) 221--229.

\bibitem[KM88]{Kenku_Momose} Kenku, M. A. and Momose, Fumiyuki,
\textit{Torsion points on elliptic curves defined over quadratic fields},
Nagoya Math. J. \textbf{109} (1988), 125--149.

\bibitem[LMFDB]{LMFDB} The LMFDB Collaboration, \textit{The $L$-functions and modular forms database}, {\it Varieties: Elliptic curves over~$\bbq(\alpha)$}, http://www.lmfdb.org, 2022, [Online; accessed 27 September 2022].

\bibitem[Ma78]{Mazur} Mazur, Barry,
\textit{Modular curves and the Eisenstein ideal. With an appendix by Mazur and M. Rapoport}
Inst. Hautes {\' E}tudes Sci. Publ. Math. No. \textbf{47} (1977), 33--186 (1978).

\bibitem[Me96]{Merel} Merel, Lo{\" i}c,
\textit{Bornes pour la torsion des courbes elliptiques sur les corps de nombres. (French) [Bounds for the torsion of elliptic curves over number fields]}
Invent. Math. \textbf{124} (1996), no. 1-3, 437--449.

\bibitem[Na10]{Najman2} Najman, Filip, \textit{Complete classification of torsion of elliptic curves over quadratic cyclotomic fields}, J. Number Theory,
Volume \textbf{130}, Issue 9 (2010),1964--1968.

\bibitem[Na11]{Najman} Najman, Filip,
\textit{Torsion of elliptic curves over quadratic cyclotomic fields},
Math. J. Okayama Univ. \textbf{53} (2011), 75--82.


\bibitem[Se97]{Serre97} Serre, Jean-Pierre,
\textit{Lectures on the Mordell-Weil theorem. Translated from the French and edited by Martin Brown from notes by Michel Waldschmidt. With a foreword by Brown and Serre. 3rd ed.}
(Vieweg \& Sohn, Braunschweig, 1997).

\bibitem[Se08]{Serre08} Serre, Jean-Pierre,
\textit{Topics in Galois theory. Second edition. With notes by Henri Darmon. Research Notes in Mathematics, 1. A K Peters, Ltd.}(Wellesley, MAm 2008).

\bibitem[Si09]{Silverman} Silverman, Joseph H.,
\textit{The arithmetic of elliptic curves, 2nd ed.}(Springer, Dordrecht, 2009).

\bibitem[DR]{DR} Rapoport, M. and Deligne, P.
\textit{Les schémas de modules de courbes elliptiques.(French)Modular functions of one variable, II}
(Proc. Internat. Summer School, Univ. Antwerp, Antwerp, 1972), pp. 143--316

\bibitem[Tr20]{Antonela} Trbovi{\' c}, Antonela,
\textit{Torsion groups of elliptic curves over quadratic fields $\bbq(\sqrt{d})$, $0<d<100$},
Acta Arith. \textbf{192} (2020), no. 2, 141--153.

\bibitem[Zh]{Zhao} Zhao, Alan, \textit{The frequency of elliptic curves over $\bbq[i]$ with fixed torsion}, preprint 2020, \href{https://arxiv.org/abs/2009.08998}{arxiv:2009.08998}

\bibitem[Zy]{Zywina} Zywina, David, \textit{Elliptic curves with maximal Galois action on their torsion points},
Bull. Lond. Math. Soc. \textbf{42} (2010), no. 5, 811--826.


\bibitem[SE]{StackExchnge} \textit{Mathematics Stack Exchange}. https://math.stackexchange.com/questions/4382512

\end{thebibliography}
\end{document}